\documentclass[11pt,reqno]{amsart}
\usepackage{indentfirst, amssymb,amsmath,amsthm, mathrsfs,cite,graphicx,float}
\newtheorem*{theoA}{Theorem A}
\newtheorem*{theoB}{Theorem B}
\newtheorem*{theoC}{Theorem C}
\newtheorem*{theoD}{Theorem D}
\newtheorem*{theoE}{Theorem E}
\newtheorem*{theoF}{Theorem F}
\newtheorem*{theoG}{Theorem G}
\newtheorem*{theoH}{Theorem H}
\newtheorem*{theoI}{Theorem I}

\newtheorem{theo}{Theorem}[section]
\newtheorem{lem}{Lemma}[section]

\newtheorem{rem}{Remark}[section]

\newtheorem{que}{Question}[section]

\newtheorem{open problem}{Open problem}[section]
\newcommand{\pa}{\partial}
\newcommand{\ol}{\overline}
\newcommand{\be}{\begin{equation}}
\newcommand{\ee}{\end{equation}}
\newcommand{\bs}{\begin{small}}
\newcommand{\es}{\end{small}}
\newcommand{\beas}{\begin{eqnarray*}}
\newcommand{\eeas}{\end{eqnarray*}}
\newcommand{\bea}{\begin{eqnarray}}
\newcommand{\eea}{\end{eqnarray}}
\renewcommand{\epsilon}{\varepsilon}
\numberwithin{equation}{section}

\begin{document}
\title[The Bohr's phenomenon]{The Bohr's phenomenon for certain K-quasiconformal harmonic mappings}
\author[V. Allu, R. Biswas and R. Mandal]{Vasudevarao Allu, Raju Biswas and Rajib Mandal}
\date{}
\address{Vasudevarao Allu, Department of Mathematics, School of Basic Science, Indian Institute of Technology Bhubaneswar, Bhubaneswar-752050, Odisha, India.}
\email{avrao@iitbbs.ac.in}
\address{Raju Biswas, Department of Mathematics, Raiganj University, Raiganj, West Bengal-733134, India.}
\email{rajubiswasjanu02@gmail.com}
\address{Rajib Mandal, Department of Mathematics, Raiganj University, Raiganj, West Bengal-733134, India.}
\email{rajibmathresearch@gmail.com}
\maketitle

\let\thefootnote\relax
\footnotetext{2020 Mathematics Subject Classification: 30A10, 30B10, 30C45, 30C62, 30C80.}
\footnotetext{Key words and phrases: Harmonic mappings, Bohr radius, Improved Bohr radius, Bohr-Rogosinski inequality, $K$-quasiconformal mappings, Concave Univalent functions, Subordination.}
\begin{abstract}
The primary objective of this paper is to establish several sharp versions of improved Bohr inequalities, refined Bohr inequalities, and Bohr-Rogosinski inequalities for the class of 
$K$-quasiconformal sense-preserving harmonic mappings $f=h+\ol{g}$ in the unit disk $\Bbb{D}:=\{z\in\mathbb{C}: |z|<1\}$, where $f$ and $g$ are analytic functions in $\mathbb{D}$.
\end{abstract}
\section{\bf Introduction and Preliminaries}
Let $f$ be an analytic function on the open unit disk $\mathbb{D}:=\{z\in\mathbb{C}:|z|<1\}$ with the Taylor series expansion  
\bea\label{e1} f(z)=\sum_{n=0}^{\infty} a_nz^n\eea 
such that $|f(z)|\leq 1$ in $\Bbb{D}$.
Then, for the majorant series $M_f(r):=\sum_{n=0}^{\infty}|a_n|r^n$ of $f$, we have
\bea\label{e2} M_f(r)\leq 1\quad\text{for}\quad|z|=r\leq 1/3.\eea
The radius $1/3$ can't be improved. It is observed that if a complex-valued function $f:\mathbb{D} \to \mathbb{C}$ satisfies the inequality $|f(z)|\le 1$ for all $z\in\mathbb{D}$, 
and if there exists a point $y\in\mathbb{D}$ such that $|f(y)|=1$, then $f(z)$ reduces to a unimodular constant function.
Here $1/3$ is known as Bohr radius while the inequality (\ref{e2}) is known as classical Bohr inequality. Actually, H. Bohr \cite{B1914} obtained the inequality (\ref{e2}) for $r\leq 1/6$ 
but later Weiner, Riesz and Schur \cite{D1995} independently improved it to $1/3$.\\[2mm] 
\indent In recent years, there has been a substantial body of research devoted to the classical Bohr inequality and its generalized forms.
The notion of Bohr radius has been discussed by Abu-Muhanna and Ali \cite{A2010,AA2011} in the context of analytic functions from $\Bbb{D}$ to simply connected domains, as well as to 
the exterior of the closed unit disk in $\Bbb{C}$. Furthermore, the Bohr phenomenon for shifted disks and simply connected domains is explored in depth in \cite{AAH2022,EPR2021,FR2010}. In their 
respective studies, Allu and Halder \cite{AH2021} and Bhowmik and Das \cite{BD2018} have examined the Bohr phenomenon in the context of subordination classes.
Boas and Khavinson \cite{BK1997} have studied the notion of the Bohr radius to encompass the case of several complex variables and identified the multidimensional Bohr 
radius as a key contribution to this field of research. Many researchers have built upon this foundation, extending and generalizing this phenomenon in different contexts 
(see \cite{A2000,AAD2001,LP2021}). We refer to \cite{AKP2019,AAH2023,1AH2022,2AH2022,BK2004,EPR2019,HLP2020,IKP2020,KKP2021,KP2018,LP2023,LPW2020,MBG2024,BM2024} and the references 
listed therein for an in-depth investigation on several other aspects of Bohr's inequality.\\[2mm]
\indent In addition to the notion of the Bohr radius, there is another concept, known as the Rogosinski radius \cite{R1923}, which is defined as follows: Let $f(z)=\sum_{n=0}^{\infty}a_nz^n$ be analytic in $\mathbb{D}$ such that $|f(z)|<1$ in $\Bbb{D}$. Then, for every $N\geq 1$, we have 
$\left|S_N(z)\right|<1$ in the disk $|z|<1/2$, where $S_N(z):=\sum_{n=0}^{N-1}a_nz^n$ denotes partial sum of $f$ and 
the radius $1/2$ is sharp. Motivated by Rogosinski radius, Kayumov and Ponnusamy \cite{KP2017} have considered the Bohr-Rogosinski sum $R_N^f(z)$ which is defined by
\beas R_N^f(z):=|f(z)|+\sum_{n=N}^{\infty}|a_n||z|^n.\eeas
It is easy to see that $|S_N(z)|=\left|f(z)-\sum_{n=N}^{\infty}a_nz^n\right|\leq R_N^f(z)$. Moreover, the Bohr-Rogosinski sum
$R_N^f(z)$ is related to the classical Bohr sum (Majorant series) in which $N=1$ and $f(0)$ is replaced by $f(z)$.
For an analytic function $f$ in $\Bbb{D}$ with $|f(z)|<1$ in $\Bbb{D}$, Kayumov and Ponnusamy \cite{KP2017} defined the Bohr-Rogosinski radius as the largest number $r\in(0, 1)$
such that $R_N^f(z)\leq 1$ for $|z|< r$ and established the following result.
\begin{theoA}\cite{KP2017}
Let $f(z)=\sum_{n=0}^{\infty} a_nz^n$ be analytic in $\mathbb{D}$ and $|f(z)|\leq 1$. Then, for each $N\in\mathbb{N}$, we have 
\bea\label{f5} |f(z)|+\sum_{n=N}^{\infty}|a_n||z|^n\leq 1\eea
for $|z|=r\leq R_N$, where $R_N$ is the positive root of the equation $2(1+r)r^N-(1-r)^2=0$. The radius $R_N$ is the best possible. 
\end{theoA}
\noindent Note that the following is another way of writing the right side of the inequality (\ref{f5}): 
\beas 1=|f(0)|+1-|f(0)|=|f(0)|+d\left(f(0), \pa \Bbb{D}\right),\eeas where $d\left(f(0),\pa \mathbb{D}\right)$ denotes the Euclidean 
distance between $f(0)$ and the boundary $\pa \mathbb{D}$ of $\mathbb{D}$.\\[2mm]
\indent 
In 2020, Ponnusamy {\it et al.} \cite{PVW2020} obtained the following refined Bohr inequality by employing a refined version of the coefficient inequalities.
\begin{theoB}\cite{PVW2020}  Let $f(z)=\sum_{n=0}^{\infty} a_nz^n$ be analytic in $\mathbb{D}$ and $|f(z)|\leq 1$. Then
\bea\label{f3}\sum_{n=0}^\infty |a_n|\rho^n+\left(\frac{1}{1+|a_0|}+\frac{\rho}{1-\rho}\right)\sum_{n=1}^\infty |a_n|^2 \rho^{2n}\leq 1\eea
for $r\leq 1/(2+|a_0|)$ and the numbers $1/(1+|a_0|)$ and $1/(2+|a_0|)$ cannot be improved.\end{theoB}
Let $h$ be an analytic function in $\mathbb{D}$ and $\mathbb{D}_r=:\{z\in\mathbb{C}: |z|<r \}$ for $0<r<1$. Let $S_r(h)$ denote the planar integral  
\beas S_r(h)=\int_{\mathbb{D}_r} |h'(z)|^2 dA(z).\eeas
If $h(z)=\sum_{n=0}^\infty a_n z^n$, then it is well-known that $S_r(h)/\pi=\sum_{n=1}^\infty n|a_n|^2 r^{2n}$ and if $h$ is univalent, then $S_r(h)$ is the area of the image $h(\mathbb{D}_r)$.
\subsection{\bf The Bohr phenomenon for the class of subordinations}
Let $\mathcal{B}$ denote the class of all analytic functions $f(z)=\sum_{n=0}^\infty a_n z^n$ that are defined on the unit disk $\mathbb{D}$, with 
$|f(z)|\leq 1$ in $\mathbb{D}$.
For any 
 analytic functions $f$ and $g$ defined on $\mathbb{D}$, we say that the function $f$ is subordinate to $g$, denoted by $f\prec g$, if there exists an $\omega\in\mathcal{B}$ with $\omega(0)=0$ and $f(z)=g\left(\omega(z)\right)$ for $z\in\mathbb{D}$. It is well-known that, if $g$ is univalent in $\mathbb{D}$, then $f\prec g$ if, and only if, $f(0)=g(0)$ and $f(\mathbb{D})\subset g(\mathbb{D})$. Also, it is clear that, if $f\prec g$, then $|f'(0)|\leq |g'(0)|$. For basic details and results on subordination classes, we refer to \cite[Chapter 6]{D1983}.\\[2mm] 
 \indent Let  $\mathcal{A}$ denote the class of all analytic functions on $g$ on $\Bbb{D}$ satisfying the normalization $g(0)=g'(0)-1=0$. Let $\mathcal{S}$ denote the class of univalnet ({\it i.e.}, one-to-one) functions on $\Bbb{D}$. Let $\mathcal{S}^*$ and $\mathcal{C}$ denote the subclass of $\mathcal{S}$ of mappings that map 
 $\mathbb{D}$ onto starlike and convex domains, respectively. For additional information regarding these classes and several other related subclasses of $\mathcal{S}$, we refer to \cite{D1983}.
If $g\in\mathcal{S}$, then $1/4\leq d\left(0,\pa g(\mathbb{D})\right)\leq 1$ and if $g\in\mathcal{C}$, then $1/2\leq d\left(0,\pa g(\mathbb{D})\right)\leq 1$ (see \cite[Theorem 2.3 and Theorem 2.15]{D1983}).\\[2mm]
\indent The concept of the Bohr phenomenon, as it applies to the family of functions defined by subordination, was first introduced by Abu-Muhanna \cite{A2010}. 
In this paper, we will use the notation $S(g)$ to denote the class of all analytic functions $f$ in $\Bbb{D}$ that are subordinate to a fixed univalent function $g$, {\it i.e.,} $S(g):=\{f: f\prec g\}$. 
The family $S(g)$ is said to exhibit a Bohr phenomenon if there exists an $r_g\in(0, 1]$ such that, for any $f(z)=\sum_{n=0}^\infty a_n z^n\in S(g)$, the inequality
\bea\label{f1}\sum_{n=1}^\infty |a_n|r^n\leq d\left(g(0),\pa g(\mathbb{D})\right)\quad\text{holds for}\quad |z|=r\leq r_g.\eea
The largest $r_g$ is called the Bohr radius. 
Similarly, the family $S(g)$ is said to exhibit a Bohr-Rogosinski phenomenon if there exists an $r_{N,g}\in(0,1]$ such that, for any $f(z)=\sum_{n=0}^\infty a_n z^n\in S(g)$, then 
\bea\label{f2} |f(z)|+\sum_{n=N}^\infty |a_n|r^n\leq |g(0)|+d\left(g(0),\pa g(\mathbb{D})\right)\quad\text{holds for}\quad |z|=r\leq r_{N,g}.\eea
The largest $r_{N,g}$ is called the Bohr-Rogosinski radius.\\[2mm]
\indent It is evident that $g(z)=(a-z)/(1-\ol{a}z)$ with $|a|<1$, then $g(\mathbb{D})=\mathbb{D}$, $S(g)=\mathcal{B}$ and $d\left(g(0),\pa g(\mathbb{D})\right)=1-|g(0)|=1-|a|$. This implies that (\ref{f1}) holds for $|z|<1/3$ by (\ref{e2}) and (\ref{f2}) holds with $r_{N,g}=R_N$, according to \textrm{Theorem A}.\\[2mm]
In \cite{A2010}, Abu-Muhanna have established the following Bohr phenomenon for the subordination class $S(g)$, where $g$ is univalent in $\mathbb{D}$.
\begin{theoC}\cite[Theorem 1, P. 1072]{A2010}
If $f(z)=\sum_{n=0}^\infty a_n z^n\in S(g)$ and $g(z)=\sum_{n=0}^\infty b_n z^n$ is univalent, then 
\beas \sum_{n=1}^\infty |a_n|r^n\leq d\left(g(0), \pa g(\Bbb{D})\right)\quad\text{for}\quad |z|=r\leq r_0=3-2\sqrt{2}\eeas
and this radius is sharp for the Koebe function $f(z)=z/(1-z)^2$. 
\end{theoC}
It has been shown that $r_0$ can be improved to $1/3$, which is sharp if $g$ is a convex and univalent function, {\it i.e.,} $g(\Bbb{D})$ is a convex domain.
It can be observed that whenever a function $f$ maps $\Bbb{D}$ into a domain $g(\Bbb{D})$ that is distinct from $\Bbb{D}$, the Bohr inequality (\ref{f1}) can be established in a 
general sense if $f\in S(g)$, where $g$ is the covering map from $\Bbb{D}$ onto $g(\Bbb{D})$ that satisfies $f(0)=g(0)$. The implementation of this concept has yielded a number of results (see \cite{A2010,AA2011,AANH2014}).
It is therefore of significant interest to study the Bohr phenomenon for $S(g)$, where $g$ belongs to a specific class of univalent functions.\\[2mm]
\indent 
The refined version of Bohr's inequality (\ref{f3}) can be expressed as follows:
\beas\sum_{n=1}^\infty |a_n|r^n+\left(\frac{1}{2-(1-|f(0)|)}+\frac{r}{1-r}\right)\sum_{n=1}^\infty |a_n|^2 r^{2n}\leq 1-|a_0|=1-|f(0)|.\eeas
It is evident that the quantity $1-|f(0)|$ represents the distance between $f(0)$ and the boundary $\pa\mathbb{D}$ of $\mathbb{D}$. Recently, Ponnusamy {\it et al.}
\cite{PVW2022,PVW2020} have derived several refined versions of Bohr's inequality for bounded analytic functions. The family $S(g)$ is said to exhibit a refined Bohr phenomenon if there exists an $r_g\in(0,1]$ such that, for any $f(z)=\sum_{n=0}^\infty a_n z^n\in S(g)$, the inequality
\beas\sum_{n=1}^\infty |a_n|r^n+\left(\frac{1}{2-\lambda}+\frac{r}{1-r}\right)\sum_{n=1}^\infty |a_n|^2 r^{2n}\leq \lambda\quad\text{holds for}\quad |z|=r\leq r_g,\eeas
where $\lambda=d\left(g(0),\pa g(\mathbb{D})\right)\leq 1$. The largest number $r_g$ is called the Bohr radius in the refined formulation.\\[2mm]
\indent 
Let us consider the complex-valued function $f=u+iv$ defined in a simply connected domain $\Omega$. If $f$ satisfies the Laplace equation $\Delta f =4f_{z\ol z} = 0$, then $f$ 
is said to be harmonic in $\Omega$, {\it i.e.}, $u$ and $v$ are real harmonic in $\Omega$. 
Note that every harmonic mapping $f$ has the canonical representation $f = h + \ol{g}$, where $h$ and $g$ are analytic in $\Omega$, known as the analytic and co-analytic parts 
of $f$, respectively, and $\ol{g(z)}$ denotes the complex conjugate of $g(z)$. This representation is unique up to an additive constant (see \cite{D2004}). The inverse function theorem 
and a result of Lewy \cite{L1936} demonstrate that a harmonic function $f$ is locally univalent in $\Omega$ if, and only if, the Jacobian  of $f$, defined by $J_f(z):=|h'(z)|^2-|g'(z)|^2$ 
is non-zero in $\Omega$. A harmonic mapping $f$ is locally univalent and sense-preserving
in $\Omega$ if, and only if, $J_f (z) > 0$ in $\Omega$ or equivalently if $h'\not=0$ in $\Omega$
and the dilatation $\omega_f:= \omega=g'/h'$ of $f$ has the property that $|\omega_f| < 1$ in $\Omega$ (see \cite{L1936}).\\[2mm]
\indent 
A locally univalent and sense-preserving harmonic mapping $f = h+\ol{g}$ is said to be $K$-quasiconformal harmonic on $\mathbb{D}$ if the condition $|\omega_f(z)| \leq k < 1$ is satisfied for $z\in\mathbb{D}$, where $K = (1+k)/(1-k) \geq 1$ (see \cite{K2008,M1968}). Clearly, $ k \to 1$ corresponds to the limiting case $K \to\infty$.\\[2mm] 
\indent 
In 2019, Liu and Ponnusamy \cite{LP2019} established the following results for determining the Bohr radii for harmonic mappings defined in the unit disk $\mathbb{D}$ whose analytic part is subordinate to some analytic function.
\begin{theoD}\cite{LP2019} Suppose that $f(z)=h(z)+\ol{g(z)}=\sum_{n=0}^\infty a_n z^n+\ol{\sum_{n=1}^\infty b_n z^n}$ is a sense-preserving $K$-quasiconformal harmonic mapping in $\mathbb{D}$ and $h\prec \phi$, where $\phi$ is univalent and convex in $\mathbb{D}$. Then
\beas\sum_{n=1}^\infty \left(|a_n|+|b_n|\right)r^n\leq d\left(\phi(0), \pa\phi(\mathbb{D})\right)\quad\text{for}\quad r\leq \frac{K+1}{5K+1}.\eeas
The number $(K+1)/(5K+1)$ is sharp.
\end{theoD}
\begin{theoE}\cite{LP2019} Suppose that $f(z)=h(z)+\ol{g(z)}=\sum_{n=0}^\infty a_n z^n+\ol{\sum_{n=1}^\infty b_n z^n}$ is a sense-preserving $K$-quasiconformal harmonic mapping in $\mathbb{D}$ and $h\prec \phi$, where $\phi$ is analytic and univalent in $\mathbb{D}$. Then
\beas\sum_{n=1}^\infty \left(|a_n|+|b_n|\right)r^n\leq d\left(\phi(0), \pa\phi(\mathbb{D})\right)\quad\text{for}\quad r\leq r_u(k),\eeas
where $k=(K-1)/(K+1)$ and $r_u(k)\in(0,1)$ is the root of the equation 
\beas (1-r)^2-4r(1+k\sqrt{1+r})=0.\eeas
\end{theoE}
In 2022, Ponnusamy {\it et al.} \cite{PVW2022} obtained the following two results pertaining to Bohr's phenomenon in a refined formulation for a more general family of subordinations.
\begin{theoF}\cite{PVW2022} Let $f(z)=\sum_{n=0}^\infty a_n z^n$ and $g$ be analytic in $\Bbb{D}$ such that $g$ is univalent and convex in $\mathbb{D}$. Assume that $f\in S(g)$ 
and $\lambda=d\left(g(0), \pa g(\mathbb{D})\right)\leq 1$. Then
\beas T_f(r):=\sum_{n=1}^\infty |a_n|r^n+\left(\frac{1}{2-\lambda}+\frac{r}{1-r}\right)\sum_{n=1}^\infty |a_n|^2 r^{2n}\leq \lambda\quad\text{holds for all}\quad r\leq r_*,\eeas
where $r_*\approx 0.24683$ is the unique positive root of the equation $3r^3-5r^2-3r+1=0$ in the interval $(0,1)$. Moreover, for any $\lambda\in(0,1)$ there exists a uniquely defined $r_0\in(r_*, 1/3)$ such that $T_f(r)\leq \lambda$ for $r\in[0,r_0]$. The radius $r_0$ can be calculated as the solution of the equation 
\beas \Phi(\lambda,r)=4r^3\lambda^2-(7r^3+3r^2-3r+1)\lambda+6r^3-2r^2-6r+2=0.\eeas
The result is sharp.\end{theoF}
\begin{theoG}\cite{PVW2022} Let $f(z)=\sum_{n=0}^\infty a_n z^n$ and $g$ be an analytic and univalent function in $\Bbb{D}$. Assume that $f\in S(g)$ 
and $\lambda=d\left(g(0), \pa g(\mathbb{D})\right)\leq 1$. Then
\beas \sum_{n=1}^\infty |a_n|r^n+\left(\frac{1}{2-\lambda}+\frac{r}{1-r}\right)\sum_{n=1}^\infty |a_n|^2 r^{2n}\leq \lambda\quad\text{holds for all}\quad r\leq r_g,\eeas
where $r_g\approx 0.128445$ is the unique positive root of the equation 
\beas (1-6r+r^2)(1-r)^2(1+r)^3-16r^2(1+r^2)=0\eeas
in the interval $(0,1)$. The result is sharp.\end{theoG}
\subsection{\bf The family of concave univalent functions with opening angle $\pi \alpha$ at infinity}
\noindent A domain $G$ is considered to have an opening angle $\pi\alpha$ at infinity if $G^c=\Bbb{C}\setminus G$ is included in a wedge-shaped region with an angle less than or equal to $\pi \alpha$ but not in a region with an angle larger than $\pi \alpha$. In the context of this study, it is necessary to introduce the following family of univalent functions:\\[1mm]
A function $g:\Bbb{D}\to\Bbb{C}$ is said to belong to the family of concave univalent functions with an opening angle $\pi \alpha$ at infinity, where $\alpha\in[1,2]$, if $g$ satisfies the following conditions:
\begin{itemize}
\item[(i)] $g\in\mathcal{A}$ and $g(1)=\infty$.
\item[(ii)] $g$ maps $\Bbb{D}$ conformally onto a set whose complement with respect to $\Bbb{C}$ is convex.
\item[(iii)] The opening angle angle of $g(\Bbb{D})$ at infinity is less than or equal to $\pi\alpha$, $\alpha\in[1,2]$.
 \end{itemize}
This family of functions is represented by $\widehat{C_0}(\alpha)$ (see \cite{AA2023,BD2018}).  For $g\in\widehat{C_0}(\alpha)$ $(\alpha\in[1,2])$, the boundary of $g(\Bbb{D})$ is
contained in a wedge shaped region with opening angle $\pi \alpha$ but not in any bigger opening angle. Further, the closed set $\Bbb{C}\setminus g(\Bbb{D})$ is convex and 
unbounded. The family of concave univalent normalized functions with opening angle $\pi \alpha$, $\alpha\in[1,2]$, at infinity is denoted by $C_0(\alpha):=\widehat{C_0}(\alpha)\cap \mathcal{S}$. In the case where $\alpha=1$, the image domain is reduced to a convex half plane. It can be observed that concave univalent functions are related to convex functions and that every $g\in\widehat{C_0}(1)$ is the convex function. For further insight into the topic of concave functions, we refer to \cite{APW2004,APW2006,AW2005,CP2007,B2012}.\\[2mm]
\indent In 2018, Bhowmik and Das \cite{BD2018} established the following result on the Bohr inequality for the family $\widehat{C_0}(\alpha)$.
\begin{theoH}\cite{BD2018}
Let $f\in\widehat{C_0}(\alpha)$, $\alpha\in[1,2]$ and $g\in S(f)$, with $f(z)=\sum_{n=0}^\infty a_nz^n$ and $g(z)=\sum_{n=0}^\infty b_nz^n$ in $\Bbb{D}$.
Then, the inequality 
\beas \sum_{n=1}^\infty |b_n|r^n\leq d\left(f(0), \pa f(\mathbb{D})\right)\quad\text{holds for}\quad |z|=r\leq r_0=(2^{1/\alpha}-1)(2^{1/\alpha}+1).\eeas
This result is sharp.
\end{theoH}
In 2023, Allu and Arora \cite{AA2023} established the following Bohr-Rogosinski inequality for the family $\widehat{C_0}(\alpha)$.
\begin{theoI}\cite{AA2023}
Let $f(z)=\sum_{n=0}^\infty a_nz^n$ and $g(z)=\sum_{n=0}^\infty b_nz^n$ be two analytic functions in $\mathbb{D}$ such that $f\in\widehat{C_0}(\alpha)$, $\alpha\in[1,2]$ and $g\in S(f)$.
Then, for each $N\in\mathbb{N}$, the inequality 
\beas |g(z)|+\sum_{n=N}^\infty |b_n|r^n\leq |f(0)|+d\left(f(0), \pa f(\mathbb{D})\right)\quad\text{holds for}\quad |z|=r<\min\{r_{\alpha,1}^N, 1/3\},\eeas
where $r_{\alpha,1}^N$ is the positive of the equation 
\beas F_{\alpha,1}^N(x)=\sum_{n=N}^\infty A_n x^n+f_\alpha(x)-\frac{1}{2\alpha}=0\eeas
in $(0,1)$ and $f_\alpha(x)=(1/(2\alpha))\left(((1+x)/(1-x))^\alpha -1\right)=\sum_{n=1}^\infty A_n x^n$. If $r_{\alpha,1}^N\leq 1/3$, then the radius $r_{\alpha,1}^N$ is the best possible.
\end{theoI}
\noindent In light of the aforementioned results, the following questions naturally arise with regard to this study.
\begin{que}\label{Q1} Can we establish the sharp improved Bohr inequality for harmonic mappings of \textrm{Theorems D} and \textrm{E} by using the non-negative quantity $S_r(h)$? \end{que}
\begin{que}\label{Q2} Can we establish the sharp improved Bohr inequality of \textrm{Theorems D} and \textrm{E} by using the concept of replacing the initial coefficients with the absolute values of the analytic function and its derivative in the majorant series?  \end{que} 
\begin{que}\label{Q3} Can we establish the refined Bohr inequality for harmonic mappings of \textrm{Theorems F} and \textrm{G}? \end{que}
\begin{que}\label{Q4} Is it possible to find the sharp Bohr radius for harmonic mappings of \textrm{Theorem H}, where the analytic component is subordinate to a function $\phi  \in\widehat{C_0}(\alpha)$? \end{que}
\begin{que}\label{Q5} Can we establish the sharp Bohr-Rogosinski inequality for harmonic mappings of \textrm{Theorem I}?  \end{que} 
The affirmative answers to Questions \ref{Q1} to \ref{Q5} are the primary purpose of this paper.\\[2mm]
The organization of the remaining part of the paper is as follows: In Section 2, we present some necessary lemmas and in Section 3, we establish two results on the refined versions of Bohr 
inequalities. In Section 4, we establish some sharp results on the improved Bohr inequalities. In Section 5, we obtain the sharp Bohr radius and Bohr-Rogosinski inequality associated with subordination to a concave univalent function. 
\section{\bf Necessary Lemmas}
The following lemmas play a vital role to prove our main results.
\begin{lem}\label{Qlem1} \cite[P. 1074-1075]{A2010} Suppose $g(z)=\sum_{n=0}^\infty a_n z^n$ is an analytic and univalent map from $\mathbb{D}$ onto a simply connected domain 
$g(\mathbb{D})$, then 
\beas \frac{1}{4}\left(1-|z|^2\right)|g'(z)|\leq d\left(g(z), \pa g(\mathbb{D})\right)\leq \left(1-|z|^2\right)|g'(z)|\quad\text{for}\quad z\in\mathbb{D}.\eeas
If $f(z)=\sum_{n=0}^\infty b_n z^n\prec g(z)$, then 
\beas |b_n|\leq n |g'(0)|\leq 4 n \;d\left(g(0), \pa g(\mathbb{D})\right)\quad\text{for}\quad n\geq 1.\eeas\end{lem}
\begin{lem}\label{Qlem2} \cite[P. 1074]{A2010} Suppose $g(z)=\sum_{n=0}^\infty a_n z^n$ is an analytic and univalent map from $\mathbb{D}$ onto a convex domain 
$g(\mathbb{D})$, then 
\beas \frac{1}{2}\left(1-|z|^2\right)|g'(z)|\leq d\left(g(z), \pa g(\mathbb{D})\right)\leq \left(1-|z|^2\right)|g'(z)|\quad\text{for}\quad z\in\mathbb{D}.\eeas
If $f(z)=\sum_{n=0}^\infty b_n z^n\prec g(z)$, then 
\beas |b_n|\leq  |g'(0)|\leq  2 \;d\left(g(0), \pa g(\mathbb{D})\right)\quad\text{for}\quad n\geq 1.\eeas\end{lem}
\begin{lem}\cite{KPS2018}\label{lem3} Suppose that $h(z)=\sum_{n=0}^\infty a_nz^n$ and $g(z)=\sum_{n=0}^\infty b_nz^n$ are two analytic functions in $\mathbb{D}$ such that $|g’(z)|\leq k|h’(z)|$ in $\mathbb{D}$ and for some $k\in [0,1)$ with $|h(z)|\leq 1$. Then,
\beas \sum_{n=1}^\infty n|b_n|^2 r^{2n}\leq k^2\sum_{n=1}^\infty n|a_n|^2 r^{2n}\;\;\text{and}\;\;
\sum_{n=1}^\infty |b_n|^2 r^n\leq k^2 \sum_{n=1}^\infty |a_n|^2 r^n\;\;\text{for}\;\;|z|=r<1.\eeas\end{lem}
\begin{lem}\label{Qlem5}\cite{GK2022} Suppose that $h(z)=\sum_{n=0}^\infty a_nz^n$ and $g(z)=\sum_{n=0}^\infty b_nz^n$ are two analytic functions in $\mathbb{D}$ and $h\prec g$, then for $N\in\Bbb{N}$, we have 
\beas \sum_{n=N}^\infty |a_n| r^n\leq \sum_{n=N}^\infty |b_n| r^n\quad\text{holds for}\quad |z|=r\leq 1/3.\eeas\end{lem}
\begin{lem}\label{Qlem6}\cite[Theorem 1]{BD2018}
Let $f\in\widehat{C_0}(\alpha)$, $\alpha\in[1,2]$ such that $f(z)=\sum_{n=0}^\infty a_n z^n$, then
\item[(i)] $|f'(0)|\leq 2\alpha \; d\left(f(0),\pa f(\Bbb{D})\right)$,
\item[(ii)] $|a_n|\leq A_n \left|f'(0)\right|$ for $n\geq 1$,
\item[(iii)] $|f(z)-f(0)|\leq |f'(0)|f_\alpha(r)$, where $f_\alpha(z)$ is defined by 
\bea\label{f4} f_\alpha(z):=\frac{1}{2\alpha}\left(\left(\frac{1+z}{1-z}\right)^\alpha-1\right)=\sum_{n=1}^\infty A_n z^n\quad\text{with}\quad A_1=1. \eea
All the inequalities are sharp for the function $f_\alpha(z)$.
\end{lem}
\begin{rem}Note that \beas \frac{1}{2\alpha}\left(\left(\frac{1+z}{1-z}\right)^\alpha-1\right)=\sum_{n=1}^\infty A_n z^n,\eeas
where $A_n\geq 0$ for $n\geq 1$, then
\beas A_1=1, A_2=\alpha, A_3=\frac{1}{3}(2\alpha^2+1), A_4=\frac{1}{3}(\alpha^3+2\alpha), A_5=\frac{1}{15}(2\alpha^4+10\alpha^2+3),\cdots.\eeas
\end{rem}
\begin{lem}\label{Qlem7}\cite{AKP2019,LPW2020} Suppose that $h(z)=\sum_{n=0}^\infty a_nz^n$ and $g(z)=\sum_{n=0}^\infty b_nz^n$ are two analytic functions in $\mathbb{D}$. If $|g'(z)|\leq k|h'(z)|$ in $\Bbb{D}$ for some $k\in(0,1]$, then
\beas \sum_{n=1}^\infty |b_n| r^n\leq k\sum_{n=1}^\infty |a_n| r^n\quad\text{for all}\quad |z|=r\leq 1/3.\eeas\end{lem}
\section{\bf Refined versions of Bohr inequalities}
In the following, we obtain the refined version of Bohr's inequality in the context of \textrm{Theorem F}  for harmonic mappings associated with subordination. 
\begin{theo}\label{T5} Suppose that $f(z)=h(z)+\ol{g(z)}=\sum_{n=0}^\infty a_n z^n+\ol{\sum_{n=1}^\infty b_n z^n}$ is a sense-preserving $K$-quasiconformal harmonic mapping in $\mathbb{D}$ and $h\prec \phi$, where $\phi$ is univalent and convex in $\mathbb{D}$. Then,
\beas S_1(r):=\sum_{n=1}^\infty |a_n| r^n+\sum_{n=1}^\infty |b_n|r^n+\left(\frac{1}{2-\lambda}+\frac{r}{1-r}\right)\left(\sum_{n=1}^\infty |a_n|^2r^{2n}+\sum_{n=1}^\infty |b_n|^2r^{2n}\right)\leq\lambda\eeas
for $r\leq r_0$, where $\lambda=d\left(\phi(0), \pa\phi(\mathbb{D})\right)\leq 1$ and $r_0\in(0,(K+1)/(5K+1))$ is the unique positive root of the equation 
\beas \left(3+2\left(\frac{K-1}{K+1}\right)\right)r^3-\left(5+4\left(\frac{K-1}{K+1}\right)^2\right)r^2-\left(3+2\left(\frac{K-1}{K+1}\right)\right)r+1=0.\eeas
Moreover, for any $r\in(r_0, (K+1)/(5K+1))$, there exists a uniquely defined $\lambda_1(r)\in(0,1)$ such that $S_1(r)\leq \lambda$ for $\lambda\leq \lambda_1(r)$, where 
$\lambda_1(r)$ satisfy the equation $G_1(\lambda_1(r),r)\\=0$, where  
\beas G_1(\lambda,r)&=&4\left(k^2+1\right)r^3\lambda^2-\left(\left(7+2k+4k^2\right)r^3+\left(3+4k^2\right)r^2-(3+2k)r+1\right)\lambda\\
&&+(6+4k)r^3-2r^2-(6+4k)r+2,\eeas 
where $k=(K-1)/(K+1)$.
\end{theo}
\begin{proof} As $h\prec \phi$ and that the function $\phi$ is univalent and convex in $\Bbb{D}$, in view of \textrm{Lemma \ref{Qlem2}}, we have 
$|a_n|\leq  |\phi'(0)|\leq  2\lambda$ for $n\geq 1$, where $\lambda=d\left(\phi(0), \pa \phi(\mathbb{D})\right)$. Therefore,
\beas\sum_{n=1}^\infty |a_n| r^n\leq 2\lambda\sum_{n=1}^\infty r^n=2 \lambda\frac{r}{1-r}\quad\text{and}
\quad\sum_{n=1}^\infty |a_n|^2 r^{2n}\leq 4\lambda^2\sum_{n=1}^\infty r^{2n}=4\lambda^2\frac{r^2}{1-r^2}.\eeas 
Since $f$ is a locally univalent and $K$-quasiconformal sense-preserving harmonic mapping on $\mathbb{D}$, Schwarz's Lemma gives that the dilatation $\omega=g'/h'$ is analytic 
in $\mathbb{D}$ and $|\omega(z)|=|g'(z)/h'(z)|\leq k$ in $\mathbb{D}$, where $K = (1+k)/(1-k) \geq 1$, $k\in[0,1)$.
In view of \textrm{Lemma \ref{lem3}}, we have
\bea\label{r4}&& \sum_{n=1}^\infty |b_n|^2 r^n\leq k^2 \sum_{n=1}^\infty |a_n|^2 r^n\leq 4k^2\lambda^2\frac{r}{1-r}\\[2mm]\text{and}
&&\label{g3}\sum_{n=1}^\infty n|b_n|^2 r^{2n}\leq k^2\sum_{n=1}^\infty n|a_n|^2 r^{2n},\eea
where $K = (1+k)/(1-k) \geq 1$, $k\in[0,1)$. 
Using (\ref{r4}) and the Cauchy-Schwarz inequality, we have 
\bea\label{r5} \sum_{n=1}^\infty |b_n| r^n\leq \left(\sum_{n=1}^\infty |b_n|^2 r^n\right)^{1/2}\left(\sum_{n=1}^\infty r^n\right)^{1/2} \leq 2 k\lambda\frac{r}{1-r}.\eea
By integrating the inequality (\ref{g3}) with respect to $r^2$ after dividing by $r^2$, we get
\beas \sum_{n=1}^\infty |b_n|^2 r^{2n}\leq k^2\sum_{n=1}^\infty |a_n|^2 r^{2n}\leq 4\lambda^2 k^2 \frac{r^2}{1-r^2}.\eeas
Therefore,
\beas S_1:&=&\sum_{n=1}^\infty |a_n| r^n+\sum_{n=1}^\infty |b_n|r^n+\left(\frac{1}{2-\lambda}+\frac{r}{1-r}\right)\left(\sum_{n=1}^\infty |a_n|^2r^{2n}+\sum_{n=1}^\infty |b_n|^2r^{2n}\right)\\[2mm]
&\leq& 2(k+1)\lambda\frac{r}{1-r}+\frac{4\lambda^2r^2(1+r-\lambda r)}{(2-\lambda)(1-r)(1-r^2)}(k^2+1)\\[2mm]
&=&\lambda+\lambda\left(\frac{(2k+3)r-1}{1-r}+\frac{4\lambda r^2(1+r(1- \lambda))(k^2+1)}{(2-\lambda)(1-r)(1-r^2)}\right)\\
&=&\lambda-\lambda\frac{G_1(\lambda,r)}{(2-\lambda)(1-r)(1-r^2)},\eeas
where 
\beas G_1(\lambda,r)&=&4\left(k^2+1\right)r^3\lambda^2-\left(\left(7+2k+4k^2\right)r^3+\left(3+4k^2\right)r^2-(3+2k)r+1\right)\lambda\\
&&+(6+4k)r^3-2r^2-(6+4k)r+2.\eeas
It is evident that $S_1>\lambda$, {\it i.e.,} $G_1(\lambda,r)<0$ for $r>1/(2k+3)$ and for each $\lambda\in(0,1]$.
Differentiating partially $G_1(\lambda,r)$ twice with respect to $\lambda$, we obtain
\beas&& \frac{\pa}{\pa \lambda}G_1(\lambda,r)=8\left(k^2+1\right)r^3\lambda-\left(\left(7+2k+4k^2\right)r^3+\left(3+4k^2\right)r^2-(3+2k)r+1\right),\\[2mm]
&& \frac{\pa^2}{\pa \lambda^2}G_1(\lambda,r)=8\left(k^2+1\right)r^3\geq 0\quad\text{for every}\quad \lambda\in(0,1].\eeas
Therefore, $\frac{\pa}{\pa \lambda}G_1(\lambda,r)$ is a monotonically increasing function of $\lambda\in(0,1]$ and it follows that
\beas  \frac{\pa}{\pa \lambda}G_1(\lambda,r)\leq \frac{\pa}{\pa \lambda}G_1(1,r)=(r-1)\left(r-r_1(k)\right)\left(r-r_2(k)\right),\eeas
where
\beas r_1(k)=\frac{1+k-\sqrt{4k-3k^2}}{4 k^2-2k+1} \quad\text{and}\quad r_2(k)=\frac{1+k+\sqrt{4k-3k^2}}{4 k^2-2k+1}.\eeas 
It is evident that $r_1(k)\in(0,1)$ and $r_2(k)\geq 1$ for $k\in[0,1)$. Furthermore, it is easy to verify that $r_1(k)>1/(2k+3)$ for $k\in[0,1)$, as illustrated in Figure \ref{fig1}.
\begin{figure}[H]
\centering
\includegraphics[scale=0.6]{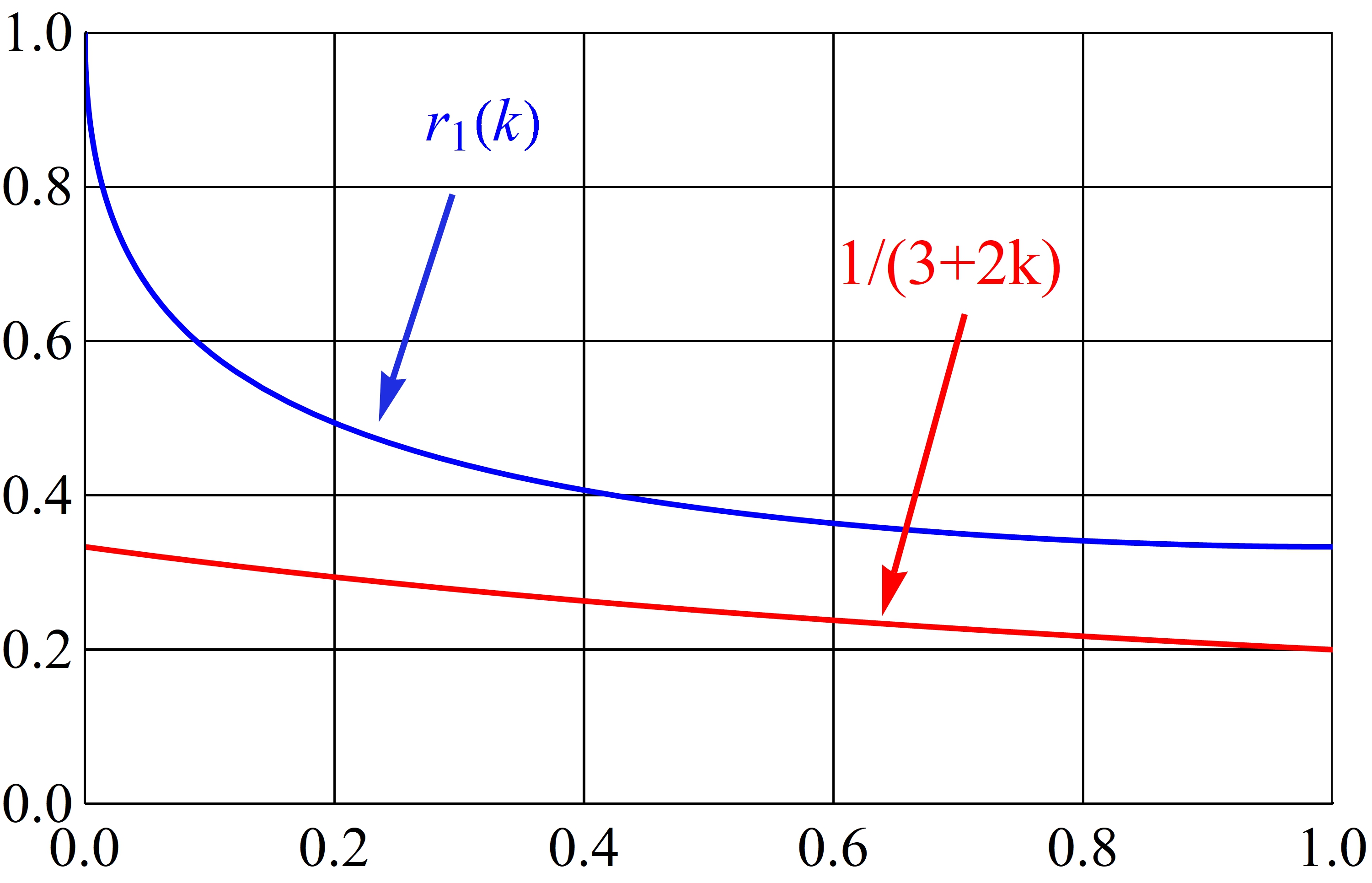}
\caption{The graphs of $r_1(k)$ and $1/(3+2k)$, when $k\in[0,1)$}
\label{fig1}
\end{figure}
\noindent Therefore, we have $\frac{\pa}{\pa \lambda}G_1(\lambda,r)\leq 0$ for $r\leq r_1(k)$.
Therefore, $G_1(\lambda,r)$ is a monotonically decreasing function of $\lambda$ for $r\leq r_1(k)$ and hence, we have
\beas G_1(\lambda,r)\geq G_1(1,r)=(3+2k)r^3-(5+4k^2)r^2-(3+2k)r+1\geq 0\quad\text{for}\quad r\leq r_0,\eeas
where $r_0\in(0,1/(3+2k))$ is the unique positive root of the equation 
\beas (3+2k)r^3-(5+4k^2)r^2-(3+2k)r+1=0,\eeas
where $k=(K-1)/(K+1)$.\\[2mm]
\indent Furthermore, $G_1(0,r)=2\left((3+2k)r-1\right)\left(r^2-1\right)$ and thus, we have $G_1(0,r)\geq 0$ for $r\leq 1/(3+2k)$ and $G_1(0,r)<0$ for $r>1/(3+2k)$. Again,
\beas G_1'(1,r)=2r\left(3(r-1)+2(kr-1)-4k^2\right)+(3+2k)\left(r^2-1\right)<0,\eeas
which shows that $G_1(1,r)\geq 0$ for $r\leq r_0$ and $G_5(1,r)<0$ for $r>r_0$. Since $G_1(\lambda,r)$ is a monotonically decreasing function of $\lambda$ for $r\leq r_1(k)$ 
and for any $r\in(r_0,1/(3+2k))$, $G_1(0,r)\geq 0$, $G_1(1,r)<0$, there is a uniquely defined $\lambda_1(r)\in(0,1)$ such that $G_1(\lambda_1(r),r)=0$, {\it i.e.,} $G_1(\lambda,r)\geq 0$ for $\lambda\leq \lambda_1(r)$. This completes the proof.
\end{proof}
\begin{rem}Setting $K=1$ in \textrm{Theorem \ref{T5}} gives \textrm{Thoerem F}.\end{rem}
The following result is the refined version of Bohr's inequality of \textrm{Theorem G} for harmonic mappings associated with subordination. 
\begin{theo}\label{T6} Suppose that $f(z)=h(z)+\ol{g(z)}=\sum_{n=0}^\infty a_n z^n+\ol{\sum_{n=1}^\infty b_n z^n}$ is a sense-preserving $K$-quasiconformal harmonic mapping in $\mathbb{D}$ and $h\prec \phi$, where $\phi$ is analytic and univalent in $\mathbb{D}$. Then,
\beas \sum_{n=1}^\infty |a_n| r^n+\sum_{n=1}^\infty |b_n|r^n+\left(\frac{1}{2-\lambda}+\frac{r}{1-r}\right)\left(\sum_{n=1}^\infty |a_n|^2r^{2n}+\sum_{n=1}^\infty |b_n|^2r^{2n}\right)\leq\lambda\eeas
for $r\leq r_0$, where $\lambda=d\left(\phi(0), \pa\phi(\mathbb{D}\right))\leq 1$ and $r_0\in\left(0, 3+2 k-2\sqrt{2+3k+k^2}\right)$ is the unique positive root of the equation 
\beas 1-5r-4kr-21r^2-4kr^2-16k^2r^2+r^3-16 r^4-16 k^2 r^4=0,\eeas
where $k=(K-1)/(K+1)$.\end{theo}
\begin{proof} Given that $h\prec \phi$ and that the function $\phi$ is analytic and univalent in $\Bbb{D}$, in view of \textrm{Lemma \ref{Qlem2}}, we have 
$|a_n|\leq n |\phi'(0)|\leq  4n\lambda$ for $n\geq 1$, where $\lambda=d\left(\phi(0), \pa \phi(\mathbb{D})\right)$. Therefore, 
\beas&&\sum_{n=1}^\infty |a_n| r^n\leq 4 \lambda\sum_{n=1}^\infty  nr^n=4 \lambda\frac{r}{(1-r)^2}\\[2mm]\text{and}
&&\sum_{n=1}^\infty |a_n|^2 r^{2n}\leq 16\lambda^2\sum_{n=1}^\infty n^2 r^{2n}=16\lambda^2\frac{r^2(1+r^2)}{(1-r^2)^3}.\eeas 
Using similar argument as in the proof of \textrm{Theorem \ref{T5}}, and in view of \textrm{Lemmas \ref{lem3}} and \ref{Qlem7}, we have the inequality (\ref{g3}) and
\beas\sum_{n=1}^\infty |b_n| r^n\leq k\sum_{n=1}^\infty |a_n| r^n \leq 4 k\lambda\frac{r}{(1-r)^2}\quad\text{for}\quad r\leq 1/3\eeas
where $K = (1+k)/(1-k) \geq 1$, $k\in[0,1)$. From (\ref{g3}), we have
\beas \sum_{n=1}^\infty |b_n|^2 r^{2n}\leq k^2\sum_{n=1}^\infty |a_n|^2 r^{2n}\leq 16\lambda^2 k^2\sum_{n=1}^\infty n^2 r^{2n}=16k^2\lambda^2\frac{r^2(1+r^2)}{(1-r^2)^3}.\eeas
Therefore,
\beas S_2:&=&\sum_{n=1}^\infty |a_n| r^n+\sum_{n=1}^\infty |b_n|r^n+\left(\frac{1}{2-\lambda}+\frac{r}{1-r}\right)\left(\sum_{n=1}^\infty |a_n|^2r^{2n}+\sum_{n=1}^\infty |b_n|^2r^{2n}\right)\\
&\leq& 4(k+1)\lambda\frac{r}{(1-r)^2}+\frac{16\lambda^2r^2(1+r^2)(1+r-\lambda r)}{(2-\lambda)(1-r)(1-r^2)^3}(k^2+1)\\[2mm]
&=&\lambda+\lambda\left(\frac{4(k+1)r-(1-r)^2}{(1-r)^2}+\frac{16\lambda r^2(1+r^2)(1+r(1- \lambda))(k^2+1)}{(2-\lambda)(1-r)(1-r^2)^3}\right)\\[2mm]
&=&\lambda+\lambda\left(\frac{(4k+6)r-1-r^2}{(1-r)^2}+\frac{16\lambda r^2(1+r^2)(1+r(1- \lambda))(k^2+1)}{(2-\lambda)(1-r)(1-r^2)^3}\right)\\[2mm]
&=&\lambda-\lambda\frac{G_2(\lambda,r)}{(2-\lambda)(1-r)(1-r^2)^3},\eeas
where 
\beas &&G_2(\lambda,r)=16(k^2+1)(r^2+1)r^3\lambda^2-R_1(r)\lambda+R_2(r)\eeas
such that 
\beas R_1(r)&=&r^7-4kr^6-5r^6+16k^2r^5-4kr^5+9r^5+16k^2r^4+8kr^4+27r^4+8kr^3\\
&&+16k^2r^3+27r^3-4kr^2+16k^2r^2+9 r^2 - 4kr-5r+1,\\[2mm]
R_2(r)&=&2-10r-8kr-14r^2-8kr^2+22r^3+16kr^3+22r^4+16kr^4-14r^5\\
&&-8kr^5-10r^6-8kr^6+2r^7.\eeas
It is evident that 
\beas (4k+6)r-1-r^2=\left(r_1(k)-r\right)\left(r-r_2(k)\right),\eeas
where $r_1(k)=3+2k-2\sqrt{2+3 k+k^2}\in(0,1/3)$ and $r_2(k)=3+2k+ 2\sqrt{2+3 k+k^2}>1$ for $k\in[0,1)$.
Therefore, $S_2>\lambda$, {\it i.e.,} $G_2(\lambda,r)<0$ for $r>r_1(k)$ and for each $\lambda\in(0,1]$.
Differentiating partially $G_2(\lambda,r)$ twice with respect to $\lambda$, we obtain 
\beas \frac{\pa}{\pa \lambda}G_2(\lambda,r)&=&32(k^2+1)(r^2+1)r^3\lambda-R_1(r)\\[2mm]\text{and}\quad
\frac{\pa^2}{\pa \lambda^2}G_2(\lambda,r)&=&32(k^2+1)(r^2+1)r^3\geq 0\quad\text{for every}\quad \lambda\in(0,1].\eeas
Therefore, $\frac{\pa}{\pa \lambda}G_2(\lambda,r)$ is a monotonically increasing function of $\lambda\in(0,1]$ and it follows that
\beas  \frac{\pa}{\pa \lambda}G_2(\lambda,r)&\leq& \frac{\pa}{\pa \lambda}G_2(1,r)\\[2mm]
&=&-r^7+(4k+6)r^6+(16k^2+4k+23)r^5-(16k^2+8k+27)r^4\\
&&+(16k^2-8k+5)r^3+(-16k^2+4k-9)r^2+(4k+5)r-1\\[2mm]
&\leq &(4k+6)r_1^6(k)+(16k^2+4k+23)r_1^5(k)+(16k^2-8k+5)r_1^3(k)\\
&&+(4k+5)r_1(k)-1\eeas
for $0<r\leq r_1(k)=3+2k-2\sqrt{2+3 k+k^2}$. A tedious long calculation shows that
\beas (4k+6)r_1^6(k)+(16k^2+4k+23)r_1^5(k)+(16k^2-8k+5)r_1^3(k)+(4k+5)r_1(k)-1\\
=2\left(R_3(k)-R_4(k)\sqrt{2+3k+k^2}\right),\hspace{8cm}\eeas
where 
\beas R_3(k)&=&8192 k^7+74752 k^6+296704 k^5+672704 k^4+951408 k^3+843432 k^2\\&&+432283 k+97732,\\[2mm]
R_4(k)&=&8192k^6+62464 k^5+204032 k^4+372928 k^3+408176 k^2+253840 k+69107.\eeas
Further, a tedious long calculation shows that the inequality $R_3(k)-R_4(k)\sqrt{2+3k+k^2}\\\leq 0$ is equivalent to
\beas 6656 k^5+33600 k^4+72272 k^3+84720 k^2+51555 k+ 11074\geq 0,\eeas
which is true for $0\leq k<1$ and it is shown in Figure \ref{fig3}. Therefore, $\frac{\pa}{\pa \lambda}G_2(\lambda,r)\leq 0$ for $0<r\leq r_1(k)=3+2k-2\sqrt{2+3 k+k^2}$.
\begin{figure}[H]
\centering
\includegraphics[scale=0.6]{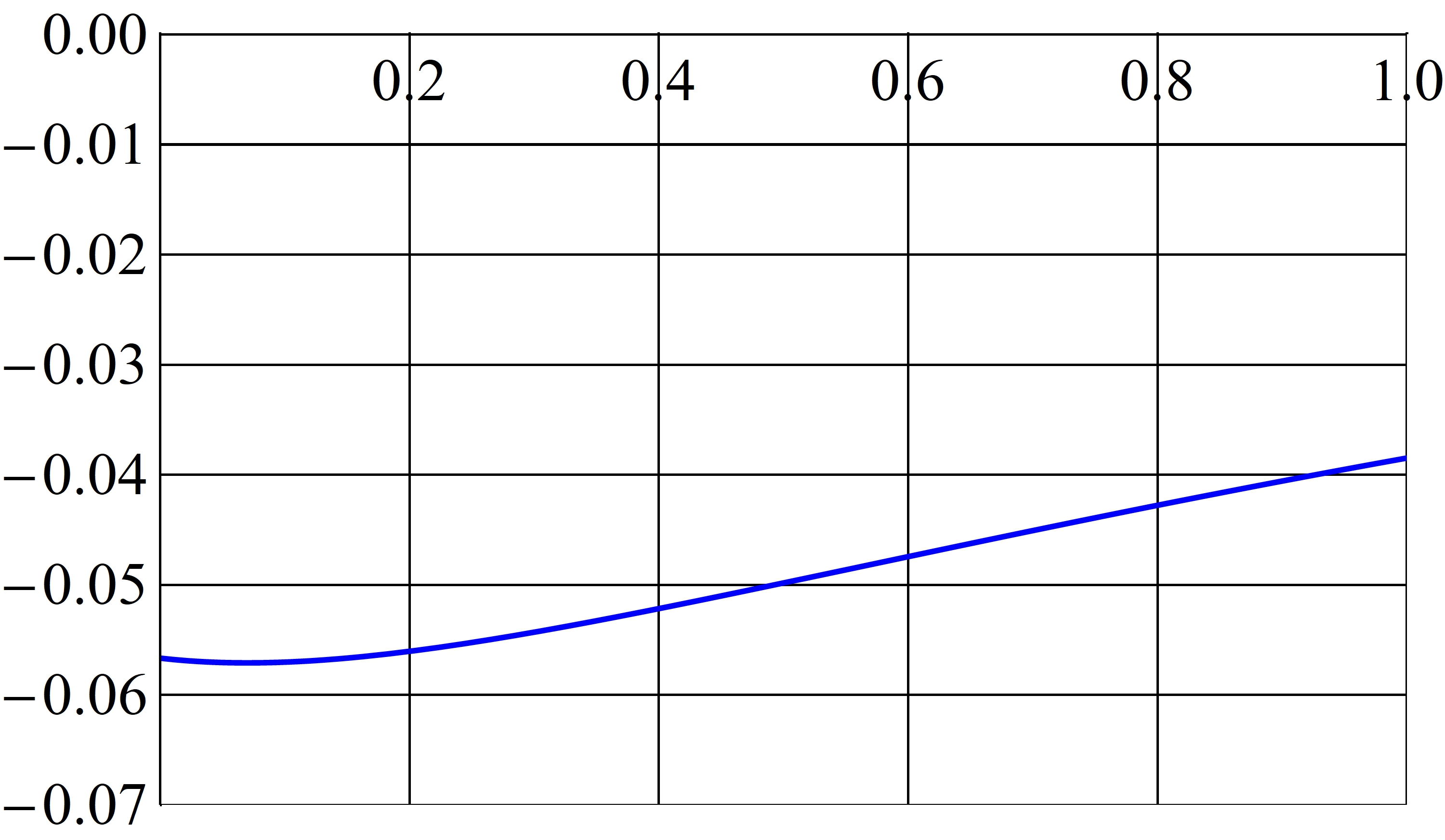}
\caption{The graph $R_3(k)-R_4(k)\sqrt{2+3k+k^2}$ for $k\in[0,1)$}
\label{fig3}
\end{figure}
\noindent Therefore, $G_2(\lambda,r)$ is a monotonically decreasing function of $\lambda$ for $r\leq r_1(k)$ and it follows that 
\beas G_2(\lambda,r)&\geq& G_2(1,r)\\
&=&r^7-4 k r^6-5 r^6-4 k r^5-7 r^5+16 k^2 r^4+27 r^4+8 k r^4+11 r^3+8 k r^3\\
&&+9 r^2-4 k r^2+16 k^2 r^2-5 r- 4 k r +1\geq 0\eeas
for $r\leq r_0$, where $r_0\in(0,r_1(k))$ is the unique positive root of the equation 
\beas &&r^7-4 k r^6-5 r^6-4 k r^5-7 r^5+16 k^2 r^4+27 r^4+8 k r^4+11 r^3+8 k r^3+9 r^2\\
&&-4 k r^2+16 k^2 r^2-5 r- 4 k r +1=0,\eeas
where $k=(K-1)/(K+1)$. Indeed,
\beas G_2'(1,r)&=&-k^2(32r+64r^3)+k(- 24 r^5- 20 r^4+12 r^3- 8 r - 4)\\
&&+(7 r^6- 30 r^5- 35 r^4+ 33 r^2- 22 r-5)-20r^3(1-k)-24r(1-kr).\eeas
It is evident that $- 24 r^5- 20 r^4+12 r^3- 8 r - 4<0$ for $r\in[0,1)$ and by employing Sturm's theorem, it is a straightforward process to ascertain that the equation 
$7 r^6- 30 r^5- 35 r^4+ 33 r^2- 22 r-5=0$ has no roots in $(0,1)$, as illustrated in Figure \ref{fig4}. 
\begin{figure}[H]
\centering
\includegraphics[scale=0.6]{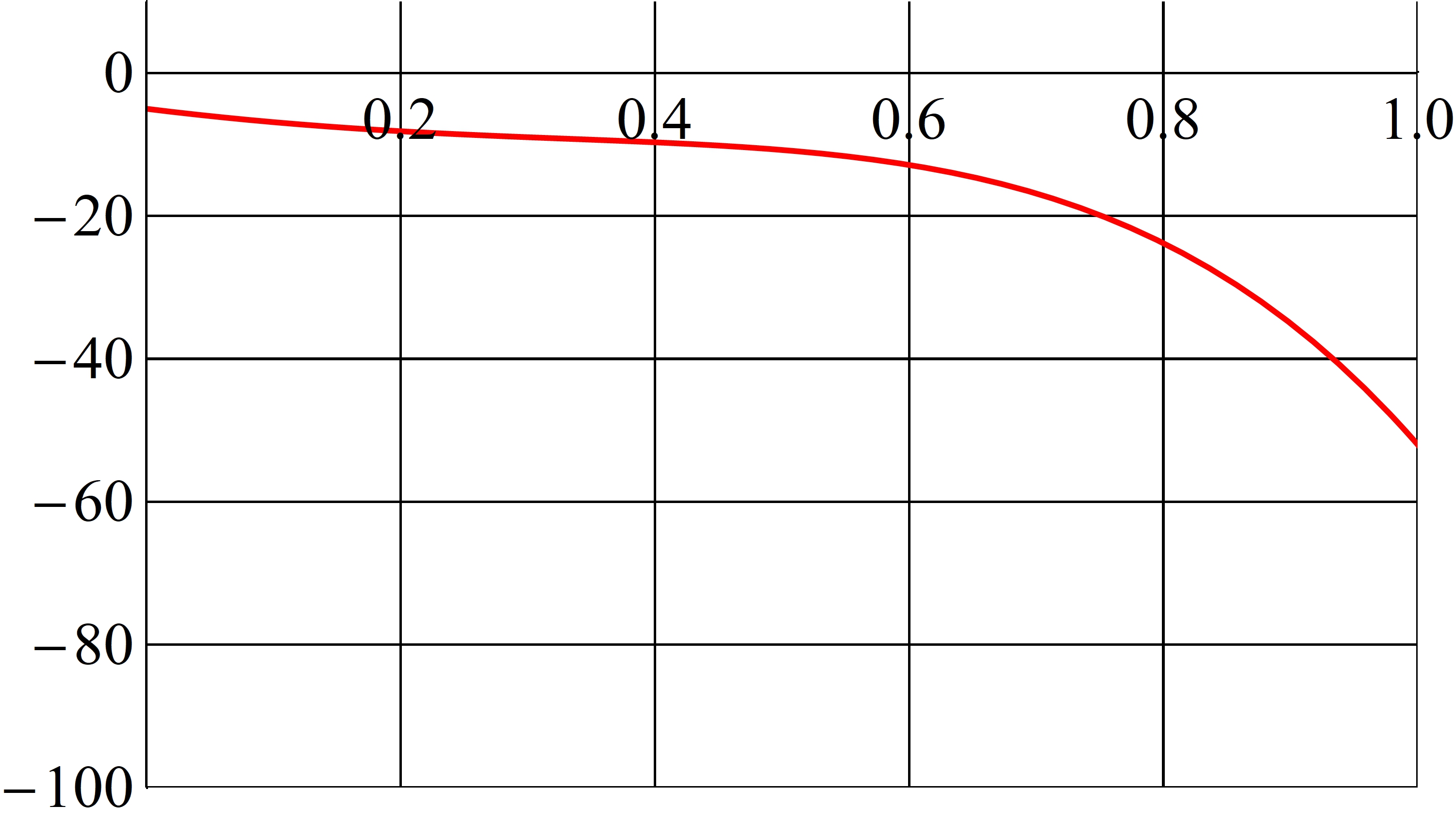}
\caption{The graph $7 r^6- 30 r^5- 35 r^4+ 33 r^2- 22 r-5$ for $r\in(0,1)$}
\label{fig4}
\end{figure}
Therefore, $G_2'(1,r)\leq 0$ for $r\in[0,1)$, {\it i.e.,} $G_2(1,r)$ is a monotonically decreasing function of $r$, 
which shows that $G_2(1,r)\geq 0$ for $r\leq r_0$ and $G_2(1,r)<0$ for $r>r_0$. This completes the proof.
\end{proof}
\begin{rem}Setting $K=1$ in \textrm{Theorem \ref{T6}} gives \textrm{Thoerem G}.\end{rem}
\section{\bf Improved version of Bohr inequalities}
In the following, we obtain the sharp improved version of Bohr inequality of \textrm{Theorem E} for harmonic mapping by making use of the non-negative quantity $S_r(h)$. 
\begin{theo}\label{T2} Suppose that $f(z)=h(z)+\ol{g(z)}=\sum_{n=0}^\infty a_n z^n+\ol{\sum_{n=1}^\infty b_n z^n}$ is a sense-preserving $K$-quasiconformal harmonic mapping in $\mathbb{D}$ and $h\prec \phi$, where $\phi$ is analytic and univalent in $\mathbb{D}$. Then
\beas\sum_{n=1}^\infty \left(|a_n|+ |b_n|\right) r^n+\mu\left(\frac{S_r(h)}{\pi}\right)^{1/2}\leq d\left(\phi(0), \pa\phi(\mathbb{D})\right)\quad\text{for}\quad r\leq r_0,\eeas
where $\mu\geq0$ and $r_0\in(0,1/3)$ is the unique positive root of the equation 
\beas (1-r^2)^2-\left(\frac{8K}{K+1}\right)r(1+r)^2-4\mu r\sqrt{r^4+4r^2+1}=0.\eeas
The number $r_0$ is the best possible.
\end{theo}
\begin{proof} 
Given that $h\prec\phi$ and that the function $\phi(z)$ is univalent and analytic in $\Bbb{D}$, the \textrm{Lemma \ref{Qlem1}} gives us
$|a_n|\leq  n|\phi'(0)|\leq  4n \;d\left(\phi(0), \pa \phi(\mathbb{D})\right)$ for $n\geq 1$. Therefore,  
\bea\label{r2}\sum_{n=1}^\infty |a_n| r^n\leq 4 \;d\left(\phi(0), \pa \phi(\mathbb{D})\right)\sum_{n=1}^\infty  nr^n=4 \;d\left(\phi(0), \pa \phi(\mathbb{D})\right)\frac{r}{(1-r)^2}.\eea
Since $f$ is a locally univalent and $K$-quasiconformal sense-preserving harmonic mapping on $\mathbb{D}$, using similar arguments as in the proof of \textrm{Theorem \ref{T5}} and in view of \textrm{Lemma \ref{Qlem7}}, we have
\bea\label{r3} \sum_{n=1}^\infty |b_n| r^n\leq k\sum_{n=1}^\infty |a_n|r^n\leq 4 k\; d\left(\phi(0), \pa \phi(\mathbb{D})\right)\frac{r}{(1-r)^2}\quad\text{for}\quad r\leq 1/3,\eea 
where $K = (1+k)/(1-k) \geq 1$, $k\in[0,1)$.
From the definition of $S_r(h)$, we have
\beas\frac{S_r(h)}{\pi}=\sum_{n=1}^\infty n|a_n|^2 r^{2n}&\leq&16\left(d\left(\phi(0), \pa \phi(\mathbb{D})\right)\right)^2\sum_{n=1}^\infty n^3 r^{2n}\nonumber\\
&=&16\left(d\left(\phi(0), \pa \phi(\mathbb{D})\right)\right)^2 \frac{r^2(r^4+4r^2+1)}{(1-r^2)^4}.\eeas
Therefore,
\beas &&\sum_{n=1}^\infty \left(|a_n|+|b_n|\right) r^n+\mu\left(\frac{S_r(h)}{\pi}\right)^{1/2}\\[2mm]
&\leq& \left( \frac{4(1+k)r}{(1-r)^2}+\frac{4\mu r\sqrt{r^4+4r^2+1}}{(1-r^2)^2} \right)d\left(\phi(0), \pa \phi(\mathbb{D})\right)\\[2mm]
&=&\frac{4(1+k)r(1+r)^2+4\mu r\sqrt{r^4+4r^2+1}}{(1-r^2)^2}d\left(\phi(0), \pa \phi(\mathbb{D})\right)\\[2mm]
&\leq& d\left(\phi(0), \pa \phi(\mathbb{D})\right)\quad \text{for}\quad r\leq r_0,\eeas
where $r_0$ is the positive root of the equation 
\bea\label{g4} G_3(r):=(1-r^2)^2-4(1+k)r(1+r)^2-4\mu r\sqrt{r^4+4r^2+1}=0,\eea
where $k=(K-1)/(K+1)$. It is easy to see that 
\beas G_3'(r)=\frac{-4\mu\left(3r^4+8r^2+1\right)}{\sqrt{r^4+4r^2+1}}-4(1+5r+3r^2-r^3+k(1+4r+3r^2))<0,\eeas
which shows that $G_3(r)$ is a strictly decreasing function of $r\in(0,1)$ with $G_3(0)=1$ and $\lim_{r\to (1/3)^{-}}G_3(r)=-4(32+3\sqrt{118}a+48k)/81<0$. Thus, the equation $G_3(r)=0$ has a unique positive root $r_0\in(0,1/3)$.\\[2mm] 
\indent To prove the sharpness of the result, we consider the function $f_1(z)=h_1(z)+\ol{g_1(z)}$ in $\mathbb{D}$ such that 
\beas \phi(z)=h_1(z)=\frac{z}{(1-z)^2}=\sum_{n=1}^\infty nz^n\quad\text{and}\quad g_1(z)=k\lambda \sum_{n=1}^\infty n z^n,\eeas
where $|\lambda|=1$ and $k=(K-1)/(K+1)$. It is evident that $d\left(\phi(0), \pa \phi(\mathbb{D})\right)=1/4$.
Thus,
\beas\sum_{n=1}^\infty \left(|a_n|+ |k\lambda a_n|\right)r^n+\mu\left(\frac{S_r(h_1)}{\pi}\right)^{1/2}
&=&\frac{(1+k)r}{(1-r)^2}+\mu\left(\sum_{n=1}^\infty n|a_n|^2 r^{2n}\right)^{1/2}\\[2mm]
&=&\frac{(1+k) r\left(1+r\right)^2+\mu r\sqrt{r^4+4r^2+1}}{(1-r^2)^2}> 1/4\eeas
for $r>r_0$, where $r_0$ is the unique positive root of the equation (\ref{g4}).											
This completes the proof.\end{proof}
\begin{rem}
Note that if we use \textrm{Lemma \ref{lem3}} instead of \textrm{Lemma \ref{Qlem7}} in the proof of \textrm{Theorem \ref{T2}}, then we have
\beas\sum_{n=1}^\infty |b_n|^2 r^n\leq k^2 \sum_{n=1}^\infty |a_n|^2 r^n&\leq &16k^2\left(d\left(\phi(0), \pa \phi(\mathbb{D})\right)\right)^2\sum_{n=1}^\infty n^2 r^n\nonumber\\&=&16k^2\left(d\left(\phi(0), \pa \phi(\mathbb{D})\right)\right)^2\frac{r(1+r)}{(1-r)^3},\eeas
 where $K = (1+k)/(1-k) \geq 1$, $k\in[0,1)$. 
By applying the Cauchy-Schwarz inequality, we have
\beas \sum_{n=1}^\infty |b_n| r^n\leq \left(\sum_{n=1}^\infty |b_n|^2 r^n\right)^{1/2}\left(\sum_{n=1}^\infty r^n\right)^{1/2} \leq 4 k\; d\left(\phi(0), \pa \phi(\mathbb{D})\right)\frac{r\sqrt{1+r}}{(1-r)^2}.\eeas 
Therefore,
\beas &&\sum_{n=1}^\infty \left(|a_n|+|b_n|\right) r^n+\mu\left(\frac{S_r(h)}{\pi}\right)^{1/2}\\[2mm]
&\leq&\frac{4(1+k\sqrt{r+1})r(1+r)^2+4\mu r\sqrt{r^4+4r^2+1}}{(1-r^2)^2}d\left(\phi(0), \pa \phi(\mathbb{D})\right)\\[2mm]
&\leq& d\left(\phi(0), \pa \phi(\mathbb{D})\right)\quad \text{for}\quad r\leq r_0,\eeas
where $r_0$ is the positive root of the equation 
\beas G_4(r):=(1-r^2)^2-4(1+k\sqrt{r+1})r(1+r)^2-4\mu r\sqrt{r^4+4r^2+1}=0,\eeas
where $k=(K-1)/(K+1)$. It is easy to see that 
\beas G'_4(r)=-\frac{2}{\sqrt{r^4+4r^2+1}}\left(\mu(2+16r^2+6r^4)+(2+10r+6r^2-2r^3)\sqrt{r^4+4r^2+1}\right.\\
\left.+k(2+9r+7r^2) \sqrt{(1+r)(r^4+4r^2+1)} \right)<0,\eeas 
which shows that $G_4(r)$ is a strictly decreasing function of $r\in(0,1)$ with 
$G_4(0)=1$ and $\lim_{r\to (1/3)^{-}}G_4(r)=-4(32 + 3\sqrt{118}a+32\sqrt{3}k)/81<0$. Thus, the equation 
$G_4(r)=0$ has a unique positive root $r_0\in(0,1/3)$. 
Since $r_0<1/3$, it is obvious that the application of \textrm{Lemma \ref{Qlem7}} in the proof of \textrm{Theorem \ref{T2}} is necessary to ensure the sharpness of the radius.
\end{rem}
The following result is the sharply improved version of the Bohr inequality of \textrm{Theorem D} for harmonic mapping by utilizing the non-negative quantity $S_r(h)$.
\begin{theo}\label{T1} Suppose that $f(z)=h(z)+\ol{g(z)}=\sum_{n=0}^\infty a_n z^n+\ol{\sum_{n=1}^\infty b_n z^n}$ is a sense-preserving $K$-quasiconformal harmonic mapping in $\mathbb{D}$ and $h\prec \phi$, where $\phi$ is univalent and convex in $\mathbb{D}$. Then
\beas\sum_{n=1}^\infty \left(|a_n|+ |b_n|\right) r^n+\mu\left(\frac{S_r(h)}{\pi}\right)^{1/2}\leq d\left(\phi(0), \pa\phi(\mathbb{D})\right)\quad\text{for}\quad r\leq r_0,\eeas
where $\mu\geq0$ and $r_0\in(0,1)$ is the unique positive root of the equation 
\beas \left(\frac{5K+1}{K+1}\right)r^2+2\left(\mu+\frac{2K}{K+1}\right)r-1=0.\eeas
The number $r_0$ is sharp.
\end{theo}
\begin{proof} 
Given that $h\prec\phi$ and that the function $\phi(z)$ is univalent and convex in $\Bbb{D}$, the \textrm{Lemma \ref{Qlem2}} gives us
$|a_n|\leq  |\phi'(0)|\leq  2 \;d\left(\phi(0), \pa \phi(\mathbb{D})\right)$ for $n\geq 1$, where $d\left(\phi(0), \pa \phi(\mathbb{D})\right)$ is the Euclidean distance between 
$\phi(0)$ and the boundary $\pa\phi(\Bbb{D})$ of $\phi(\Bbb{D})$. Therefore, 
\bea\label{r6}\sum_{n=1}^\infty |a_n| r^n\leq 2 \;d\left(\phi(0), \pa \phi(\mathbb{D})\right)\sum_{n=1}^\infty r^n=2 \;d\left(\phi(0), \pa \phi(\mathbb{D})\right)\frac{r}{1-r}.\eea
By similar arguments as in the proof of the \textrm{Theorem \ref{T5}} and in view of \textrm{Lemma \ref{lem3}}, we have the inequality (\ref{r5}), where $K = (1+k)/(1-k) \geq 1$, $k\in[0,1)$.
From the definition of $S_r(h)$, we have
\beas \frac{S_r(h)}{\pi}=\sum_{n=1}^\infty n|a_n|^2 r^{2n}\leq 4\left(d\left(\phi(0), \pa \phi(\mathbb{D})\right)\right)^2\sum_{n=1}^\infty n r^{2n}=\frac{4\left(d\left(\phi(0), \pa \phi(\mathbb{D})\right)\right)^2 r^2}{(1-r^2)^2}.\eeas
Therefore,
\beas \sum_{n=1}^\infty \left(|a_n|+|b_n|\right) r^n+\mu\left(\frac{S_r(h)}{\pi}\right)^{1/2}&\leq& 
\left( 2(k+1)\frac{r}{1-r}+2\mu\frac{r}{1-r^2} \right)d\left(\phi(0), \pa \phi(\mathbb{D})\right)\\[2mm]
&=&\frac{2(k+1)(r+r^2)+2\mu r}{1-r^2}d\left(\phi(0), \pa \phi(\mathbb{D})\right)\\[2mm]
&\leq &d\left(\phi(0), \pa \phi(\mathbb{D})\right)\quad \text{for}\quad r\leq r_0,\eeas
where $r_0$ is the positive root of the equation 
\bea\label{g6} G_5(r):=(2k+3)r^2+2(\mu+k+1)r-1=0,\eea
where $k=(K-1)/(K+1)$. It is evident that the function $G_5(r)$ is a strictly increasing function of $r\in[0,1)$ with $G_5(0)=-1<0$ and $\lim_{r\to 1^{-}}G_5(r)=4k+2\mu+4>0$. 
Therefore, the equation $G_5(r)=0$ has a unique root $r_0\in(0,1)$. \\[2mm]
\indent To prove the sharpness of the result, we consider the function $f_2(z)=h_2(z)+\ol{g_2(z)}$ in $\mathbb{D}$ such that 
\beas \phi(z)=h_2(z)=\frac{1}{1-z}=\sum_{n=0}^\infty z^n\quad\text{and}\quad g_2(z)=k\lambda \sum_{n=1}^\infty z^n,\eeas
where $|\lambda|=1$ and $k=(K-1)/(K+1)$. It is evident that $d\left(\phi(0), \pa \phi(\mathbb{D})\right)=1/2$.
Thus,
\beas \sum_{n=1}^\infty \left(|a_n|+ |k\lambda a_n|\right)r^n+\mu\left(\frac{S_r(h_2)}{\pi}\right)^{1/2}
&=&(1+k)\frac{r}{1-r}+\mu\left(\frac{r^2}{(1-r^2)^2}\right)^{1/2}\\[2mm]
&=&\frac{(1+k) \left(r^2+r\right)+\mu r}{1-r^2}> 1/2\eeas
for $r>r_0$, where $r_0$ is the positive root of the equation (\ref{g6}).										
This shows that $r_0$ is the best possible. This completes the proof.
\end{proof}
The following two results are the sharp improved versions of Bohr inequality of \textrm{Theorems D} and \textrm{E}, by the concept of replacing $|a_0|$ with $|h(z)|$ and $|a_1|$ with $|h'(z)|$, respectively, in the majorant series.
\begin{theo}\label{T3} Suppose that $f(z)=h(z)+\ol{g(z)}=\sum_{n=0}^\infty a_n z^n+\ol{\sum_{n=1}^\infty b_n z^n}$ is a sense-preserving $K$-quasiconformal harmonic mapping in $\mathbb{D}$ and $h\prec \phi$, where $\phi$ is univalent and convex in $\mathbb{D}$. Then
\beas |h(z)|+|h'(z)| r+\sum_{n=2}^\infty |a_n| r^n+\sum_{n=1}^\infty |b_n| r^n\leq |\phi(0)|+d\left(\phi(0), \pa\phi(\mathbb{D})\right)\quad\text{for}\quad r\leq r_0,\eeas
where $r_0\in(0,1)$ is the smallest positive root of the equation 
\beas (1-r)^2-2\left(r^2+\left(\frac{2K}{K+1}\right)r\right)(1-r)-2r=0.\eeas
The number $r_0$ is sharp.
\end{theo}
\begin{proof} 
Given that $h\prec\phi$ and that the function $\phi(z)$ is univalent and convex in $\Bbb{D}$, the \textrm{Lemma \ref{Qlem2}} gives us
$|a_n|\leq  |\phi'(0)|\leq  2 \;d\left(\phi(0), \pa \phi(\mathbb{D})\right)$ for $n\geq 1$. Then, we have the inequality (\ref{r6}) and
\beas\sum_{n=2}^\infty |a_n| r^n\leq 2 \;d\left(\phi(0), \pa \phi(\mathbb{D})\right)\sum_{n=2}^\infty r^n=2 \;d\left(\phi(0), \pa \phi(\mathbb{D})\right)\frac{r^2}{1-r}.\eeas 
Using similar arguments as in the proof of the \textrm{Theorem \ref{T5}} and in view of \textrm{Lemma \ref{lem3}}, we have the inequality (\ref{r5}),
where $K = (1+k)/(1-k) \geq 1$, $k\in[0,1)$. 
As $\phi$ is univalent and $h\prec \phi$, it follows that $h(0)=\phi(0)$ and 
\beas |h(z)|\leq |h(0)|+\left|\sum_{n=1}^\infty a_n z^n\right|\leq |\phi(0)|+2 d\left(\phi(0), \pa \phi(\mathbb{D})\right)\frac{r}{1-r}.\eeas  
It is easy to see that
\beas |h'(z)|=\left|\sum_{n=1}^\infty n a_n z^{n-1}\right|\leq 2 d\left(\phi(0), \pa \phi(\mathbb{D})\right)\sum_{n=1}^\infty n r^{n-1}=2 d\left(\phi(0), \pa \phi(\mathbb{D})\right)\frac{1}{(1-r)^2}.\eeas
Therefore,
\beas &&|h(z)|+|h'(z)| r+\sum_{n=2}^\infty |a_n| r^n+\sum_{n=1}^\infty |b_n| r^n\\[2mm]
&&\leq|\phi(0)|+\frac{2\left(r^2+(k+1)r\right)(1-r)+2r}{(1-r)^2}d\left(\phi(0), \pa \phi(\mathbb{D})\right)\\[2mm]
&&\leq |\phi(0)|+d\left(\phi(0), \pa \phi(\mathbb{D})\right)\quad \text{for}\quad r\leq r_0,\eeas
where $r_0\in(0,1)$ is the smallest positive root of the equation 
\bea\label{g7}(1-r)^2-2\left(r^2+(k+1)r\right)(1-r)-2r=0,\eea
where $k=(K-1)/(K+1)$.\\[2mm]
\indent To prove the sharpness of the result, we consider the function $f_3(z)=h_3(z)+\ol{g_3(z)}$ in $\mathbb{D}$ such that 
\beas \phi(z)=h_3(z)=\frac{z}{1-z}=\sum_{n=1}^\infty z^n\quad\text{and}\quad g_3(z)=k\lambda \sum_{n=1}^\infty z^n,\eeas
where $|\lambda|=1$ and $k=(K-1)/(K+1)$. It is evident that $d\left(\phi(0), \pa \phi(\mathbb{D})\right)=1/2$ and $|\phi(0)|=0$.
Thus,
\beas &&|h_3(r)|+|h_3'(r)| r+\sum_{n=2}^\infty |a_n| r^n+\sum_{n=1}^\infty |k\lambda a_n| r^n\\[2mm]
&=&\frac{r}{(1-r)}+\frac{r}{(1-r)^2}+\frac{r^2}{(1-r)}+\frac{kr}{(1-r)}> 1/2\quad\text{for}\quad r>r_0,\eeas
where $r_0$ is the positive root of the equation (\ref{g7}).
This shows that $r_0$ is the best possible. This completes the proof.
\end{proof}
\begin{theo}\label{T4}  Suppose that $f(z)=h(z)+\ol{g(z)}=\sum_{n=0}^\infty a_n z^n+\ol{\sum_{n=1}^\infty b_n z^n}$ is a sense-preserving $K$-quasiconformal harmonic mapping in $\mathbb{D}$ and $h\prec \phi$, where $\phi$ is analytic and univalent in $\mathbb{D}$. Then
\beas |h(z)|+\left|h'(z)\right|r+\sum_{n=2}^\infty \left|a_n\right|r^n+\sum_{n=1}^\infty |b_n|r^n\leq |\phi(0)|+d\left(\phi(0), \pa\phi(\mathbb{D})\right)\quad\text{for}\quad r\leq r_0,\eeas
where $r_0\in(0, 1/3)$ is the unique positive root of the equation 
\beas (1-r)^3-4\left(r^2(2-r)+\left(\frac{2K}{K+1}\right)r\right)(1-r)-4r(1+r)=0.\eeas
The number $r_0$ is sharp.
\end{theo}
\begin{proof} Given that $h\prec \phi$ and the function $\phi$ is analytic and univalent in $\Bbb{D}$, in light of \textrm{Lemma \ref{Qlem2}}, we have 
$|a_n|\leq  n|\phi'(0)|\leq  4n \;d\left(\phi(0), \pa \phi(\mathbb{D})\right)$ for $n\geq 1$. Therefore,
\beas\sum_{n=2}^\infty |a_n| r^n\leq 4 \;d\left(\phi(0), \pa \phi(\mathbb{D})\right)\sum_{n=2}^\infty nr^n=4 \;d\left(\phi(0), \pa \phi(\mathbb{D})\right)\frac{r^2(2-r)}{(1-r)^2}.\eeas 
Using similar argument as in the proof of the \textrm{Theorem \ref{T2}} and in view of \textrm{Lemma \ref{Qlem7}}, we have the inequality (\ref{r3}),
where $K = (1+k)/(1-k) \geq 1$, $k\in[0,1)$. 
As $\phi$ is univalent and $h\prec \phi$, it follows that $h(0)=\phi(0)$ and 
\beas |h(z)|\leq |h(0)|+\sum_{n=1}^\infty |a_n| r^n\leq |\phi(0)|+4 d\left(\phi(0), \pa \phi(\mathbb{D})\right)\frac{r}{(1-r)^2}.\eeas  
It is evident that
\beas |h'(z)|=\left|\sum_{n=1}^\infty n a_n z^{n-1}\right|\leq 4 d\left(\phi(0), \pa \phi(\mathbb{D})\right)\sum_{n=1}^\infty n^2 r^{n-1}=4 d\left(\phi(0), \pa \phi(\mathbb{D})\right)\frac{1+r}{(1-r)^3}.\eeas
Therefore,
\beas &&|h(z)|+|h'(z)| r+\sum_{n=2}^\infty |a_n| r^n+\sum_{n=1}^\infty |b_n| r^n\\[2mm]
&\leq&|\phi(0)|+\frac{4\left(r^2(2-r)+(1+k)r\right)(1-r)+4r(1+r)}{(1-r)^3}d\left(\phi(0), \pa \phi(\mathbb{D})\right)\\[2mm]
&\leq& |\phi(0)|+d\left(\phi(0), \pa \phi(\mathbb{D})\right)\quad \text{for}\quad r\leq r_0,\eeas
where $r_0$ is the positive root of the equation 
\bea\label{g5}G_6(r):=(1-r)^3-4\left(r^2(2-r)+(1+k)r\right)(1-r)-4r(1+r)=0,\eea
where $k=(K-1)/(K+1)$. Therefore,
\beas G_6'(r)&=&-11-4k-10 r+8kr+33r^2-16 r^3\\
&=&-4k(1-r)-4r(1-r)-\left(11+6r+16r^3-33r^2\right).\eeas
By employing Sturm's theorem, it's a straightforward process to ascertain that the equation $11+6r+16r^3-33r^2=0$ has no root in the interval $(0,1)$. Consequently, we conclude that $11+6r+16r^3-33r^2>0$ in this interval, as illustrated in Figure \ref{fig2}.
\begin{figure}[H]
\centering
\includegraphics[scale=0.7]{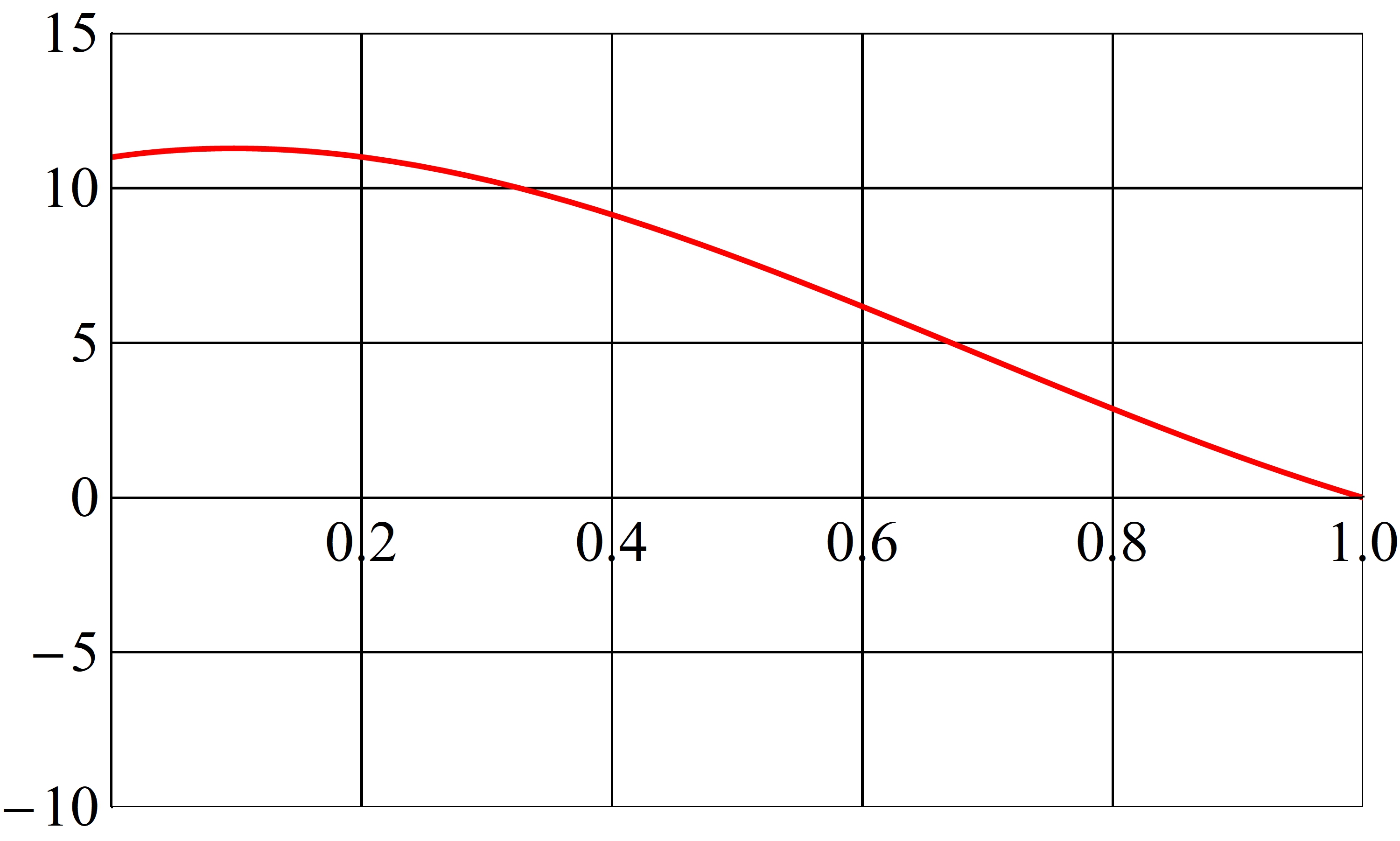}
\caption{The graph of the polynomial $11+6r+16r^3-33r^2$ in $(0,1)$}
\label{fig2}
\end{figure}
Therefore, $G'_6(r)<0$ for $r\in(0,1)$, which shows that $G_6(r)$ is a strictly decreasing function of $r\in(0,1)$ with $G_6(0)=1$ and $\lim_{r\to (1/3)^{-}}G_6(r)=-4(29+9k)/81<0$. Thus, the equation 
$G_6(r)=0$ has a unique positive root $r_0\in(0,1/3)$.\\[2mm]
\indent To prove the sharpness of the result, we consider the function $f_4(z)=h_4(z)+\ol{g_4(z)}$ in $\mathbb{D}$ such that 
\beas \phi(z)=h_4(z)=\frac{z}{(1-z)^2}=\sum_{n=1}^\infty nz^n\quad\text{and}\quad g_4(z)=k\lambda \sum_{n=1}^\infty n z^n,\eeas
where $|\lambda|=1$ and $k=(K-1)/(K+1)$. It is evident that $d\left(\phi(0), \pa \phi(\mathbb{D})\right)=1/4$ and $|\phi(0)|=0$.
Thus,
\beas &&|h_4(r)|+|h_4'(r)| r+\sum_{n=2}^\infty |a_n| r^n+\sum_{n=1}^\infty |k\lambda a_n| r^n\\[2mm]
&=&\frac{r}{(1-r)^2}+\frac{r(1+r)}{(1-r)^3}+\frac{r^2(2-r)}{(1-r)^2}+\frac{kr}{(1-r)^2}> 1/4\quad\text{for}\quad r>r_0\eeas
where $r_0$ is the positive root of the equation (\ref{g5}).
This completes the proof.
\end{proof}
\section{\bf Bohr phenomenon related to the family of concave univalent functions}
In the following, we obtain the sharp Bohr radius for harmonic mappings in which the analytic part is subordinated to a function belonging to the family $\widehat{C_0}(\alpha)$, $\alpha\in[1,2]$.  
\begin{theo}\label{T7} Suppose that $f(z)=h(z)+\ol{g(z)}=\sum_{n=0}^\infty a_n z^n+\ol{\sum_{n=1}^\infty b_n z^n}$ is a sense-preserving $K$-quasiconformal harmonic mapping in $\mathbb{D}$ and $h\prec \phi$, where $\phi\in\widehat{C_0}(\alpha)$ and $\alpha\in[1,2]$. Then
\beas\sum_{n=1}^\infty \left(|a_n|+ |b_n|\right) r^n\leq d\left(\phi(0), \pa\phi(\mathbb{D})\right)\quad\text{for}\quad r\leq r_0,\eeas
where $r_0\in(0,1/3)$ is the unique root of the equation 
\beas \left(\frac{2K}{K+1}\right)\left(\left(\frac{1+r}{1-r}\right)^\alpha-1\right)-1=0.\eeas
The number $r_0$ is sharp.
\end{theo}
\begin{proof}
As $h\prec \phi$ and  $\phi\in\widehat{C_0}(\alpha)$, $\alpha\in[1,2]$ with $\phi(z)=\sum_{n=0}^\infty c_n z^n$, thus, in view of \textrm{Lemmas \ref{Qlem5}} and \ref{Qlem6}, we have 
\bea\label{j2} \sum_{n=1}^\infty |a_n| r^n\leq \sum_{n=1}^\infty |c_n| r^n\leq 2\alpha\;d\left(\phi(0), \pa\phi(\mathbb{D})\right)\sum_{n=1}^\infty A_n r^n\quad\text{for}\quad |z|=r\leq 1/3.\eea
Using similar arguments as in the proof of \textrm{Theorem \ref{T5}}, and in view of \textrm{Lemma \ref{Qlem7}}, we have 
\bea\label{j3} \sum_{n=1}^\infty |b_n| r^n\leq k \sum_{n=1}^\infty |a_n| r^n\leq  2k\alpha\;d\left(\phi(0), \pa\phi(\mathbb{D})\right)\sum_{n=1}^\infty A_n r^n\quad\text{for}\quad |z|=r\leq 1/3.\eea
Therefore,
\beas S_3(r):=\sum_{n=1}^\infty \left(|a_n|+|b_n|\right) r^n&\leq&
 2\alpha(k+1)\;d\left(\phi(0), \pa\phi(\mathbb{D})\right)\sum_{n=1}^\infty A_n r^n\\
 &=&(k+1)\left(\left(\frac{1+r}{1-r}\right)^\alpha-1\right)d\left(\phi(0), \pa\phi(\mathbb{D})\right).\eeas
Let 
\beas G_7(r)=(k+1)\left(\left(\frac{1+r}{1-r}\right)^\alpha-1\right)-1,\quad\text{where}\quad k=(K-1)/(K+1).\eeas 
Then, \beas G_7'(r)=\frac{2\alpha(k+1)(1+r)^{\alpha}}{(1-r^2) (1-r)^\alpha}> 0,\eeas
which shows that $G_7(r)$ is a strictly increasing function of $r\in[0,1)$. Moreover, $G_7(0)=-1<0$ and $\lim_{r\to 1/3}G_7(r)=(2^r-2) + (2^r-1) k>0$. Thus, the equation 
$G_7(r)=0$ has a unique positive root $r_0\in(0,1/3)$. 
Hence, $S_3(r)\leq d\left(\phi(0), \pa \phi(\mathbb{D})\right)$ for $r\leq r_0$, where $r_0$ is the positive root of the equation $G_7(r)=0$.
\\[2mm]
\indent To prove the sharpness of the result, we consider the function $f_5(z)=h_5(z)+\ol{g_5(z)}$ in $\mathbb{D}$ such that 
\beas \phi(z)=h_5(z)=\frac{1}{2\alpha}\left(\left(\frac{1+z}{1-z}\right)^\alpha-1\right)=\sum_{n=1}^\infty A_n z^n \quad\text{and}\quad g_5(z)=k\lambda \sum_{n=1}^\infty A_n z^n,\eeas
where $|\lambda|=1$ and $k=(K-1)/(K+1)$. It is evident that $d\left(\phi(0), \pa \phi(\mathbb{D})\right)=1/(2\alpha)$.
Thus,
\beas \sum_{n=1}^\infty \left(|A_n|+ |k\lambda A_n|\right)r^n=(1+k)\sum_{n=1}^\infty A_n r^n
&=&\frac{(1+k)}{2\alpha}\left(\left(\frac{1+r}{1-r}\right)^\alpha-1\right)\\[2mm]&>& \frac{1}{2\alpha}=d\left(\phi(0), \pa \phi(\mathbb{D})\right)\quad\text{for}\quad r>r_0,\eeas
where $r_0$ is the positive root of the equation $G_7(r)=0$.
This shows that $r_0$ is the best possible. This completes the proof.
\end{proof}
In the following, we obtain the Bohr-Rogosinski inequality for harmonic mappings in which the analytic part is subordinated to a function belonging to the family $\widehat{C_0}(\alpha)$, $\alpha\in[1,2]$. 
\begin{theo}\label{T8} Suppose that $f(z)=h(z)+\ol{g(z)}=\sum_{n=0}^\infty a_n z^n+\ol{\sum_{n=1}^\infty b_n z^n}$ is a sense-preserving $K$-quasiconformal harmonic mapping in $\mathbb{D}$ and $h\prec \phi$, where $\phi\in\widehat{C_0}(\alpha)$ and $\alpha\in[1,2]$. Then
\beas|h(z)|+\sum_{n=1}^\infty \left(|a_n|+ |b_n|\right) r^n\leq |\phi(0)|+d\left(\phi(0), \pa\phi(\mathbb{D})\right)\quad\text{for}\quad r\leq r_0,\eeas
where $r_0\in(0,1/3)$ is the unique root of the equation 
\beas \left(\frac{3K+1}{K+1}\right)\left(\left(\frac{1+r}{1-r}\right)^\alpha-1\right)-1=0.\eeas
The number $r_0$ is sharp.
\end{theo}
\begin{proof} Using similar arguments as in the proof of \textrm{Theorem \ref{T7}} and in view of \textrm{Lemmas \ref{Qlem5}}, \ref{Qlem6} and \ref{Qlem7}, we have the inequalities (\ref{j2}) and (\ref{j3}). Since $h\prec \phi$, where $\phi\in\widehat{C_0}(\alpha)$ and $\alpha\in[1,2]$, thus 
\beas |h(z)|\leq |h(0)|+|h(z)-h(0)|&=&|\phi(0)|+\left|\sum_{n=1}^\infty a_n z^n\right|\\
&\leq &|\phi(0)|+2\alpha\;d\left(\phi(0), \pa\phi(\mathbb{D})\right)\sum_{n=1}^\infty A_n r^n\eeas
 for $|z|=r\leq 1/3$. Therefore,
\beas S_4(r):&=&|h(z)|+\sum_{n=1}^\infty \left(|a_n|+ |b_n|\right) r^n\\[2mm]
&\leq&|\phi(0)|+2\alpha(k+2)\;d\left(\phi(0), \pa\phi(\mathbb{D})\right)\sum_{n=1}^\infty A_n r^n\\[2mm]
 &=&|\phi(0)|+(k+2)\left(\left(\frac{1+r}{1-r}\right)^\alpha-1\right)d\left(\phi(0), \pa\phi(\mathbb{D})\right).\eeas
Let 
\beas G_8(r)=(k+2)\left(\left(\frac{1+r}{1-r}\right)^\alpha-1\right)-1,\quad\text{where}\quad k=(K-1)/(K+1).\eeas 
Then, \beas G_8'(r)=\frac{2\alpha(k+2)(1+r)^{\alpha}}{(1-r^2) (1-r)^\alpha}> 0,\eeas
which shows that $G_8(r)$ is a strictly increasing function of $r\in[0,1)$. Moreover, $G_8(0)=-1<0$ and $\lim_{r\to 1/3}G_8(r)=(2^{1+r}-3) + (2^r-1) k>0$. Thus, the equation $G_8(r)=0$ has a unique positive root $r_0\in(0,1/3)$. 
Hence, $S_4(r)\leq d\left(\phi(0), \pa \phi(\mathbb{D})\right)$ for $r\leq r_0$, where $r_0$ is the positive root of the equation $G_8(r)=0$.
\\[2mm]
\indent To prove the sharpness of the result, we consider the function $f_6(z)=h_6(z)+\ol{g_6(z)}$ in $\mathbb{D}$ such that 
\beas \phi(z)=h_6(z)=\frac{1}{2\alpha}\left(\left(\frac{1+z}{1-z}\right)^\alpha-1\right)=\sum_{n=1}^\infty A_n z^n \quad\text{and}\quad g_6(z)=k\lambda \sum_{n=1}^\infty A_n z^n,\eeas
where $|\lambda|=1$ and $k=(K-1)/(K+1)$. It is evident that $d\left(\phi(0), \pa \phi(\mathbb{D})\right)=1/(2\alpha)$.
Thus, for $z=r$,
\beas |h_6(r)|+\sum_{n=1}^\infty \left(|A_n|+ |k\lambda A_n|\right)r^n&=&(2+k)\sum_{n=1}^\infty A_n r^n=\frac{(2+k)}{2\alpha}\left(\left(\frac{1+r}{1-r}\right)^\alpha-1\right)\\[2mm]
&>& \frac{1}{2\alpha}=|\phi(0)|+d\left(\phi(0), \pa \phi(\mathbb{D})\right)\quad\text{for}\quad r>r_0,\eeas
where $r_0$ is the positive root of the equation $G_8(r)=0$. 
This shows that $r_0$ is the best possible. This completes the proof.
\end{proof}
\section{Declarations}
\noindent{\bf Acknowledgment:} The work of the second author is supported by University Grants Commission (IN) fellowship (No. F. 44 - 1/2018 (SA - III)).\\[1mm]
{\bf Conflict of Interest:} The authors declare that there are no conflicts of interest regarding the publication of this paper.\\[1mm]
{\bf Availability of data and materials:} Not applicable.\\[1mm]
{\bf Author's Contributions:} All authors contribute equally to the completion of the manuscript.


\begin{thebibliography}{33}
\bibitem{A2010} {\sc Y. Abu-Muhanna}, Bohr's phenomenon in subordination and bounded harmonic classes, {\it Comp. Var. Elliptic Equ.} {\bf 55}(11) (2010), 1071--1078.
\bibitem{AA2011} {\sc Y. Abu-Muhanna} and {\sc R. M. Ali}, Bohr's phenomenon for analytic functions into the exterior of a compact convex body, {\it J. Math. Anal. Appl.} {\bf 379}(2) (2011), 512--517.
\bibitem{AANH2014} {\sc Y. Abu-Muhanna, R. M. Ali, Z. C. Ng} and {\sc S. F. M. Hasni}, Bohr radius for subordinating families of analytic functions and bounded harmonic mappings, {\it J. Math. Anal. Appl.} {\bf420} (1) (2014), 124--136.
\bibitem{A2000} {\sc L. Aizenberg}, Multidimensional analogues of Bohr's theorem on power series, {\it Proc. Amer. Math. Soc.} {\bf128} (2000), 1147--1155.
\bibitem{AAD2001} {\sc L. Aizenberg, A. Aytuna} and {\sc P. Djakov}, Generalization of a theorem of Bohr for bases in spaces of holomorphic functions of several complex variables, {\it J. Math. Anal. Appl.} {\bf258} (2001), 429--447.
\bibitem{AKP2019} {\sc S. A. Alkhaleefah, I. R. Kayumov} and {\sc S. Ponnusamy}, On the Bohr inequality with a fixed zero coefficient, {\it Proc. Am. Math. Soc.} {\bf147}(12) (2019), 5263--5274.
\bibitem{AAH2022} {\sc M. B. Ahamed, V. Allu} and {\sc H. Halder}, The Bohr phenomenon for analytic functions on a shifted disk, {\it Ann. Fenn. Math.} {\bf47} (2022), 103--120.
\bibitem{AAH2023}{\sc M. B. Ahamed, V. Allu} and {\sc H. Halder}, Improved Bohr inequalities for certain class of harmonic univalent functions, {\it Complex Var. Elliptic Equ.} {\bf68} (2023), 267--290.
\bibitem{AA2023} {\sc V. Allu} and {\sc V. Arora}, Bohr-Rogosinski type inequalities for concave univalent functions, {\it J. Math. Anal. Appl.} {\bf520} (2023), 126845.
\bibitem{AH2021} {\sc V. Allu} and {\sc H. Halder}, Bohr radius for certain classes of starlike and convex univalent functions, {\it J. Math. Anal. Appl.} {\bf493}(1) (2021), 124519.
\bibitem{1AH2022} {\sc V. Allu} and {\sc H. Halder}, Bohr phenomenon for certain close-to-convex analytic functions, {\it Comput. Methods Funct. Theory} {\bf22} (2022), 491--517.
\bibitem{2AH2022}{\sc V. Allu} and {\sc H. Halder}, Bohr inequality for certain harmonic mappings, {\it Indag. Math.} {\bf33}(3) (2022), 581--597.
\bibitem{APW2004} {\sc F. G. Avkhadiev, C. Pommerenke} and {\sc K.-J. Wirths}, On the coefficients of concave univalent functions, {\it Math. Nachr.} {\bf271} (2004), 3--9.
\bibitem{APW2006}{\sc  F. G. Avkhadiev, C. Pommerenke} and {\sc K.-J. Wirths}, Sharp inequalities for the coefficient of concave schlicht functions, {\it Comment. Math. Helv.} {\bf81} (2006), 801--807.
\bibitem{AW2005}{\sc F. G. Avkhadiev} and {\sc K.-J. Wirths}, Concave schlicht functions with bounded opening angle at infinity, {\it Lobachevskii J. Math.} {\bf17} (2005), 3--10.
\bibitem{BK2004} {\sc C. B\'en\'eteau, A. Dahlner} and {\sc D. Khavinson}, Remarks on the Bohr phenomenon, {\it Comput. Methods Funct. Theory} {\bf4}(1) (2004), 1--19.
\bibitem{B2012}{\sc B. Bhowmik}, On concave univalent functions, {\it Math. Nachr. } {\bf285}(5-6) (2012), 606--612.
\bibitem{BD2018}{\sc B. Bhowmik} and {\sc N. Das}, Bohr phenomenon for subordinating families of certain univalent functions, {\it J. Math. Anal. Appl.} {\bf462}(2) (2018), 1087--1098.
\bibitem{BM2024}{\sc R. Biswas} and {\sc R. Mandal}, Generalized Bohr inequalities for $K$-quasiconformal harmonic mappings and their applications, preprint, https://doi.org/10.48550/arXiv.2411.01837.
\bibitem{BK1997} {\sc H. P. Boas} and {\sc D. Khavinson}, Bohr's power series theorem in several variables, {\it Proc. Am. Math. Soc.} {\bf125}(10) (1997), 2975--2979.
\bibitem{B1914} {\sc H. Bohr}, A theorem concerning power series, {\it Proc. London Math. Soc.} {\bf13}(2) (1914), 1--5.
\bibitem{CP2007}{\sc L. Cruz } and {\sc C. Pommerenke}, On concave univalent functions, {\it Complex Var. Elliptic Equ.} {\bf52} (2007), 153--159.
\bibitem{D1995} {\sc P. G. Dixon}, Banach algebras satisfying the non-unital von Neumann inequality, {\it Bull. Lond. Math. Soc.} {\bf27}(4) (1995), 359--362.
\bibitem{D1983} {\sc P. L. Duren}, Univalent functions, {\it Springer}, New York (1983).
\bibitem{D2004} {\sc P. L. Duren}, Harmonic mapping in the plane, {\it Cambridge University Press} (2004).
\bibitem{EPR2019} {\sc S. Evdoridis, S. Ponnusamy} and {\sc A. Rasila}, Improved Bohr's inequality for locally univalent harmonic mappings, {\it Indag. Math. (N.S.)} {\bf30} (2019), 201-213.
\bibitem{EPR2021} {\sc S. Evdoridis, S. Ponnusamy} and {\sc A. Rasila}, Improved Bohr's inequality for shifted disks, {\it Results Math.} {\bf76} (2021), 14.
\bibitem{FR2010} {\sc R. Fournier} and {\sc ST. Ruscheweyh}, On the Bohr radius for simply connected plane domains, {\it CRM Proc. Lect. Notes} {\bf51} (2010), 165-171.
\bibitem{GK2022}{\sc K. Gangania} and {\sc S. S. Kumar}, Bohr-Rogosinski phenomenon for $\mathcal{S}^*(\psi)$ and $\mathcal{C}(\psi)$, {\it Mediterr. J. Math.} {19} (2022), 161.
\bibitem{HLP2020} {\sc Y. Huang, M. -S. Liu} and {\sc S. Ponnusamy}, Refined Bohr-type inequalities with area measure for bounded analytic functions, {\it Anal. Math. Phys.} {\bf10} (2020), 50.
\bibitem{IKP2020} {\sc A. Ismagilov, I. R. Kayumov} and {\sc S. Ponnusamy}, Sharp Bohr type inequality, {\it J. Math. Anal. Appl.} {\bf489} (2020), 124147.
\bibitem{K2008} {\sc D. Kalaj}, Quasiconformal harmonic mapping between Jordan domains, {\it Math. Z.} {\bf260}(2) (2008), 237--252.
\bibitem{KKP2021}  {\sc I. R. Kayumov, D. M. Khammatova} and {\sc S. Ponnusamy}, Bohr-Rogosinski phenomenon for analytic functions and Ces\'aro operators, {\it J. Math. Anal. Appl.} {\bf 493}(2) (2021), 124824. 
\bibitem{KP2017} {\sc I. R. Kayumov} and {\sc S. Ponnusamy}, Bohr-Rogosinski radius for analytic functions, preprint, https://doi.org/10.48550/arXiv.1708.05585.
\bibitem{KP2018}{\sc I. R. Kayumov} and {\sc S. Ponnusamy}, Bohr's inequalities for the analytic functions with lacunary series and harmonic functions, {\it J. Math. Anal. Appl.} {\bf465} (2018), 857--871.
\bibitem{KPS2018}{\sc I. R. Kayumov, S. Ponnusamy} and {\sc N. Shakirov}, Bohr radius for locally univalent harmonic mappings, {\it Math. Nachr.} {\bf291} (2018), 1757--1768.
\bibitem{L1936} {\sc H. Lewy}, On the non-vanishing of the Jacobian in certain one-to-one mappings, {\it Bull. Am. Math. Soc.} {\bf42} (1936), 689--692.
\bibitem{LP2023} {\sc G. Liu and S. Ponnusamy}, Improved Bohr inequality for harmonic mappings, {\it Math. Nachr.} {\bf296} (2023), 716--731.
\bibitem{LP2019} {\sc Z. H. Liu} and {\sc S. Ponnusamy}, Bohr radius for subordination and $K$-quasiconformal harmonic mappings, {\it Bull. Malays. Math. Sci. Soc.} {\bf42} (2019), 2151--2168.
\bibitem{LP2021} {\sc M. S. Liu} and {\sc S. Ponnusamy}, Multidimensional analogues of refined Bohr's inequality, {\it Proc. Amer. Math. Soc.} {\bf149} (2021), 2133--2146.
\bibitem{LPW2020} {\sc M. S. Liu, S. Ponnusamy} and {\sc J. Wang}: Bohr's phenomenon for the classes of Quasi-subordination and $K$-quasiregular harmonic mappings, {\it Rev. Real Acad. Cienc. Exactas Fis. Nat. - A: Mat.} {\bf114} (2020), 115.
\bibitem{M1968} {\sc O. Martio}, On harmonic quasiconformal mappings, {\it Ann. Acad. Sci. Fenn. A. I.} {\bf425} (1968), 3--10.
\bibitem{MBG2024} {\sc R. Mandal, R. Biswas} and {\sc S. K. Guin}, Geometric studies and the Bohr radius for certain normalized harmonic mappings, {\it Bull. Malays. Math. Sci. Soc.} {\bf47} (2024), 131.
\bibitem{PVW2020} {\sc S. Ponnusamy, R. Vijayakumar} and {\sc K. -J. Wirths}, New inequalities for the coefficients of unimodular bounded functions, {\it Results Math} {\bf 75} (2020), 107.
\bibitem{PVW2022} {\sc S. Ponnusamy, R. Vijayakumar} and {\sc K. -J. Wirths}, Improved Bohr's phenomenon in quasi-subordination classes, {\it J. Math. Anal. Appl.} {\bf506}(1) (2022), 125645.
\bibitem{R1923} {\sc W. Rogosinski}, \"Uber Bildschranken bei Potenzreihen und ihren Abschnitten, {\it Math. Z.} {\bf17} (1923), 260-276.
\end{thebibliography}
\end{document}